\documentclass[12pt]{amsart}

\usepackage{amssymb}
\usepackage{amsmath}

\newcommand{\Lt}{\mbox{$\mathfrak t$}}

\newcommand{\Pf}{{\em Proof}. }
\newcommand{\EPf}{\hfill$\square$}
\newcommand{\Z}{\mbox{$\mathbf Z$}}
\newcommand{\R}{\mbox{$\mathbf R$}}
\newcommand{\C}{\mbox{$\mathbf C$}}
\newcommand{\Q}{\mbox{$\mathbf H$}}

\newcommand{\SU}[1]{\mbox{$\mathbf{SU}(#1)$}}
\newcommand{\U}[1]{\mbox{$\mathbf{U}(#1)$}}
\newcommand{\SP}[1]{\mbox{$\mathbf{Sp}(#1)$}}
\newcommand{\SO}[1]{\mbox{$\mathbf{SO}(#1)$}}
\newcommand{\OG}[1]{\mbox{$\mathbf{O}(#1)$}}
\newcommand{\Spin}[1]{\mbox{$\mathbf{Spin}(#1)$}}
\newcommand{\E}[1]{\mbox{$\mathbf{E}_{#1}$}}
\newcommand{\F}{\mbox{$\mathbf{F}_4$}}
\newcommand{\G}{\mbox{$\mathbf{G}_2$}}

\newcommand{\tref}[1]{Theorem~\ref{#1}}
\newcommand{\cref}[1]{Corollary~\ref{#1}}
\newcommand{\pref}[1]{Proposition~\ref{#1}}

\newcommand{\lref}[1]{Lemma~\ref{#1}}

\newtheorem{thm}{Theorem}[section]
\newtheorem*{thm*}{Theorem}
\newtheorem*{thmmain*}{MAIN THEOREM}
\newtheorem{lem}[thm]{Lemma}
\newtheorem{cor}[thm]{Corollary}
\newtheorem{prop}[thm]{Proposition}
\newtheorem*{prop*}{Proposition}
 
\theoremstyle{definition}

\theoremstyle{remark}
\newtheorem{rem}{Remark}[section]
\newtheorem{ex}[rem]{Example}
\newtheorem{quest}{Question}[section]

\title{On orbit spaces of\\ 
representations of compact Lie groups}

\author{Claudio Gorodski and Alexander Lytchak}

\thanks{The first author was partially supported by the
CNPq grant 302472/2009-6 and the FAPESP project 2007/03192-7.} 

\thanks{The second author was partially supported 
by a Heisenberg grant of the DFG and by the SFB 878
{\it Groups, geometry and actions}}

\date{\today}

\begin{document}

\begin{abstract}
We investigate orthogonal representations of compact Lie groups from the 
point of view of their quotient spaces, considered as metric spaces.  
We study metric spaces which are simultaneously quotients of
different representations and investigate properties of the corresponding 
representations. We obtain some structural results and apply them
to study irreducible representations of small copolarity.
As an important tool, we classify all irreducible 
representations of connected groups with cohomogeneity four or five.

\end{abstract}

\maketitle

\section{Introduction}
\subsection{General observations}  For an orthogonal representation 
$\rho :G \to \OG{V}$  of a compact Lie group $G$, the quotient metric space 
$V/G$ is the most important invariant of the action, at least  
from the metric point of view.  It seems to be a  difficult task to extract
information about algebraic invariants from the metric quotient.  It is even worse:
some algebraic invariants   are not (!) invariants of the quotient.   To address
this issue we define         
two representations $\rho_ i:G_i  \to\OG{V_i}$, $i=1, 2$, to be   
\emph{quotient-equivalent}  if~$V_1/G_1$ 
and~$V_2/G_2$ are isometric. 

We think that the study of quotient-equivalence classes is  
interesting in its own right. Moreover, the examples and connections 
that we discuss below suggest that quotient-equi\-va\-lence classes
can help to understand the geometry of representations in general,
and provide new geometrically interesting classes of representations. 
It seems to us that the understanding and solution of the most basic 
problems will require deep and interesting insights 
into the structure of the quotients. We were  able to solve most 
basic questions only in some special  cases and 
hope that the results obtained in the present paper will be only 
a first step in this direction.

Even the most basic algebraic invariant, the dimension of the 
representation space $\dim (V)$, is not always an invariant of 
a quotient-equivalence class.
The following observation, that is  a basic step for all
subsequent considerations, gives an easy tool to study the 
dimension of representation spaces in a given 
quotient-equivalence class. 

\begin{prop} \label{veryfirst}
Let $\rho_i :G_i\to \OG {V_i}$, $i=1$, $2$, 
be two quotient-equivalent representations.  
If the quotient space $V_i /G_i$ has no boundary then 
$\dim (V_1)=\dim (V_2)$. Under the weaker assumption that  
the quotient space $V_1/G_1^0$ of the representation of the 
identity component  $G_1 ^0$  of $G_1$
has no boundary, the inequality   $\dim (V_1) \leq  \dim (V_2)$ holds true.
\end{prop}

Here, the \emph{boundary} of the quotient space is understood in the 
sense of Alexandrov spaces, i.e., it is the closure of the union
of all strata of codimension~$1$ in the quotient space 
(cf.~subsection~\ref{boundary}). Already this first observation shows 
that the non-triviality of quotient-equivalence classes is closely 
related to this intriguing property of representations: 
the presence or not of boundary in the quotient. This extremely
restrictive property has been studied by Kollross and Wilking, 
who have announced 
a classification of all representations of simple connected
groups with this property. In the case of semi-simple groups the 
problem seems to be much more involved.

Before we proceed, we present a few important families of 
examples of quotient-equivalence. 
The simplest example of quotient-equivalence is given by 
\emph{orbit-equivalence}.
Recall that representations $\rho_1$ and $\rho_2$ 
are called \emph{orbit-equivalent} if there exists
an isometry from $V_1$ to $V_2$ that maps $G_1$-orbits onto $G_2$-orbits.  
In particular, any representation
is orbit-equivalent to its effectivization.
Thus we may and will always restrict ourselves to effective representations.

The next very important family of examples is provided by the  
reduction of the principal isotropy groups, 
often used in geometry and algebra 
(cf.~\cite{lr,str,gt0}). Assume that the principal isotropy group $K$ of 
the effective representation $\rho:G\to\OG{V}$ is
non-trivial. Let $V^K$ denote the subspace of all fixed points of 
$K$ and let $N(K)$ denote the normalizer of $K$
in $G$ with its induced action on $V^K$.  
Then the inclusion  $V^K\to V$ induces an isometry $V^K/\bar N \to V/G$,
where $\bar N=N(K)/K$. Note that $\dim (V^K) < \dim (V)$.  

A family that is at the origin of all ideas and results of this paper 
is given by \emph{polar} representations
(we refer the reader not 
acquainted with this species to subsection~\ref{copolar} below,
or to~\cite{dad,pt2,bco}).  
For connected groups, these 
are representations orbit-equivalent to the isotropy representations of 
symmetric spaces and, already due to this fact,
closely tied with many interesting  geometric and algebraic objects.  
In our language, a representation is polar if and only
if it is quotient-equivalent to a representation of a finite group. 
As a main example: the isotropy representation of a symmetric space
is quotient-equivalent to the action of the corresponding Weyl group 
on a maximal infinitesimal flat.

The preceding family of examples shows in an extreme way  
that a representation of a connected group may be quotient-equivalent 
to a representation of a disconnected group. 
This is a source of  difficulties  but gives rise
to some nice geometric considerations.  
Recall that the Weyl group  of a symmetric
space is a finite Coxeter group generated by reflections.    
It turns out that 
in general the picture is not much different:

\begin{prop} \label{verysecond}
Let $\rho_i:G_i \to \OG {V_i}$, $i=1$, $2$,
be quotient-equivalent representations. 
Assume that $G_1$ is connected. Then the 
action of the finite group $G_2 /G_2 ^0$ of connected components of 
$G_2$ on $V_2 /G_2 ^0$ is 
generated by reflections at subspaces of codimension~$1$ in $V_2 /G_2 ^0$.
\end{prop} 

This observation seems to be new even for the reduction 
to the normalizer of the principal isotropy group. 
In this case \pref{verysecond} implies 
the following purely group-theoretic consequence:

\begin{cor} \label{principalcoxeter}
Let $\rho : G \to \OG {V}$ be a representation of a compact 
connected group. Let $K$ be a  principal isotropy group, let 
$N(K)$ denote the normalizer of $K$ in $G$ and set $\bar N=N(K)/K$.  
Then the group of connected components $\bar N / \bar N^0$
of $\bar N$ is a  Coxeter group.
\end{cor}

We would like to mention that \pref{verysecond} remains valid for 
isometric actions on simply connected Riemannian manifolds, 
cf.~section~\ref{secreflections}. 

\subsection{Main results} 
We have seen above that non-triviality of quotient-equivalence class 
has very strong and interesting implications on the structure 
of the corresponding quotient space and representations.  
Now, we are going to describe a program
of finding all such non-trivial classes and the main results of 
this paper, closely related to this program.  

We call a representation $\rho :G \to \OG {V}$ \emph{reduced} 
if the dimension of $G$ is minimal in the quotient-equivalence class of $\rho$.
For quotient-equivalent representations $\rho _i:G_i\to \OG {V_i}$, 
$i=1$, $2$, we will call $\rho _2$ a \emph{reduction}
of $\rho _1$ if $\dim (G_2) < \dim (G_1)$.  
If, in addition,  
$\rho _2$ is reduced it will be called a \emph{minimal reduction} 
of $\rho _1$. Note that a reduced representation $\rho :G\to \OG {V}$  
has trivial principal isotropy groups, thus  satisfies 
$\dim (G) + \dim (V/G) =\dim (V)$.   Hence, $\dim (V)$ is minimal in the quotient-equivalence class of $\rho$.

It seems reasonable to analyze quotient-equivalent classes  
of irreducible representations of connected groups
by looking for possible minimal elements in such classes, 
thus by asking the following question:
 
\begin{quest}
What  are reduced representations $\tau  :G \to \OG {V}$ 
that are non-trivial reductions of an irreducible representation of some
connected group?  
\end{quest} 

We fix the group $G$ and run through the different representations of $G$. 
The above propositions and observations impose severe restrictions. First, 
the principal isotropy group must be trivial.
Second, either the group $G$ must be connected and 
$V/G$ must have a non-trivial boundary, or $G^0$ must be normalized by 
an involution in $\OG {V}$ that acts
as a reflection on $V/G^0$.   
Both conditions are not very difficult to check for any concrete 
representation of the connected component $G^0$ of
$G$.  Moreover, it seems that (for any group $G^0$) the   
conditions can be fulfilled only for finitely many irreducible 
representations of $G^0$. It appears possible to understand  the situation 
completely for any given group $G$, however, with the growing complexity 
of $G$ the case-by-case study seems to become more and more involved.
In the case of small-dimensional $G$ we obtain a satisfactory answer, 
demonstrating how strong the restrictions are.

\begin{thm} \label{verythird}
Let $\rho  :H \to \OG {W}$ be an irreducible representation of a compact 
connected group ~$H$. Let  $\tau :G \to \OG {V}$  be a minimal 
reduction of $\rho$. If $\dim (G) \leq 6$ then $G^0$ is a
torus $T^k$. Moreover, if $k\geq 1$ then $\dim (V) =2k+2$.
\end{thm}

The case $\dim (G)=0$ in the above proposition describes polar 
representations $\rho$. In this case, $\dim (V)$, the rank 
of the corresponding symmetric space, can be arbitrary.   
Thus the second claim of the above proposition appears a bit surprising: 
It is much easier for a representation to have a reduction to a 
representation of a discrete group than to a representation of 
a group of small but positive dimension.

We discuss an interpretation of the last result in terms of \emph{copolarity}.
Recall (cf.~subsection~\ref{copolar} or~\cite{got}) 
that a \emph{generalized section} of  the representation $\rho:H\to\OG W$ 
is a subspace $V$ of $W$ that intersects at least one principal orbit and 
that contains the whole normal space to the orbit at all 
such intersections. Examples of generalized sections are given by
\emph{sections} of polar representations and sets of fixed points $W^K$
of the principal isotropy group $K$.
For any minimal generalized section $\Sigma$ of $\rho$, 
the subgroup $H_{\Sigma} := \{ h\in H | h (\Sigma )= \Sigma \}$ acts on 
the subspace $\Sigma$ with a kernel 
which we quotient out to get the effectivization $\bar H_{\Sigma}$,
and then the inclusion $\Sigma \to W$  induces  
an isometry $\Sigma /\bar H_{\Sigma} \to W/H$.     
Thus we obtain a reduction of $\rho$ to a minimal generalized section.   
For such a minimal generalized section $\Sigma$, the number 
$\dim (\Sigma )- \dim (W/H)$ is called the \emph{copolarity} 
of the representation $\rho$. Due to its minimality, 
the action of $\bar H_{\Sigma}$ has trivial principal isotropy groups, thus
the copolarity of $\rho$ equals the dimension of $\bar H_{\Sigma}$. 
We say that the  representation $\rho :H\to \OG {W} $ has \emph{trivial copolarity} if there are no generalized section but the
whole space $W$.

The following result that has motivated all our investigations is a 
consequence of \tref{verythird}.  
It has been previously obtained in  \cite{got} for~$k=1$ via submanifold theory.

\begin{cor} \label{main}
Let $\rho:H\to\OG{W}$ be a non-reduced
non-polar irreducible   representation of 
a connected compact Lie group $H$. 
If the copolarity~$k$ of the representation~$\rho$
is at most~$6$, then the cohomogeneity of $\rho$ is exactly $k+2$.
\end{cor}

The result is sharp: for~$k=0$, i.e., in the case of polar actions, 
the cohomogeneity can be arbitrary.
On the other hand, for~$k=7$, there exists  
a representation of cohomogeneity~$5$ and non-trivial 
copolarity~$7$, see \tref{mainclass} below.

Returning to \tref{verythird}, we see that the restriction of~$\tau$ to~$G^0$ 
is a reducible representation, since $G^0$ is Abelian.
The case-by-case study outlined above would be inefficient in the general
case of reductions to groups whose identity component acts reducibly, 
due to the fact that there are too many reducible representations.  
On the other hand, in this situation the next result can be applied. 
It contains the second claim of \tref{verythird} as a special case and is 
much more general and precise.
For aesthetic and technical reasons we state it in a general form, 
see the  subsequent corollary for the most 
important special case.

\begin{thm}\label{special}
Let $\rho:H\to\OG{W}$ and  $\rho':H'\to\OG{W'}$ be 
quotient-equivalent representations. Assume that the action of the 
identity component $H^0$ on $W$ is irreducible and that 
of $(H')^0$ on $W'$ is reducible. Then there is precisely
one effective 
representation $\tau:G\to\mathbf O(V)$ in the quotient
class of $\rho$ and $\rho'$ which has trivial copolarity.
If this quotient-equivalence class is non-polar,
then the identity component of $G$ is a torus $T^k$ and its action on
$V$ can be identified with that of a maximal torus of $\SU{k+1}$
on $\C^{k+1}$. 
\end{thm}

As the most important special case needed in \tref{verythird} we obtain:

\begin{cor} \label{reductible}
Let $\rho:H\to\OG{W}$ be a non-reduced
non-polar irreducible representation of a 
connected compact Lie group $H$. Let
$\tau:G \to\OG{V}$ be a minimal reduction of $\rho$. Let $G^0$ be the identity component
of $G$. If $G^0$ acts reducibly on $V$, then $G^0$ 
is a torus $T^k$ and its action on
$V$ can be identified with that of a maximal torus of $\SU{k+1}$
on $\C^{k+1}$. 
\end{cor}

We would like to mention the following classes of examples, 
showing that representations described 
by the results above  are quite numerous and interesting:

\begin{ex}
A direct classification due to Straume~\cite{str} shows that
all non-polar irreducible representations of cohomogeneity~$3$   
admit a generalized section of dimension $4$, thus a reduction
to a representation of a one-dimensional group. 
\end{ex}

\begin{ex}
Let $\hat\rho:\hat H\to\OG W$ be the
isotropy representation of an irreducible
Hermitian symmetric space of rank~$r\geq2$,
where $\hat H$ is connected. Then $\hat H=S^1\cdot H$ where 
$H$ is the semisimple factor. Assume that the restriction $\rho$ of 
$\hat\rho$ to $H$ is irreducible and 
not orbit-equivalent to~$\hat\rho$ (there are 
four families of such representations; 
the second assumption is precisely associated to   
noncompact Hermitian symmetric spaces
\emph{of tube type}, see e.g.~\cite[\S9]{clerc}). Then 
$\rho:H\to\OG W$ has non-trivial copolarity
equal to $r-1$, and a minimal generalized section is given by the 
complexification of a section of $\hat\rho$~\cite{got}. Moreover, the  induced action on the minimal section is as described in \cref{reductible}. 
\end{ex}

These examples show that the above are  
indeed statements about 
existing and interesting  objects.  In fact, it turns out
that the irreducible representations characterized in \cref{reductible}
by the means of the special behaviour of their
quotient-equivalence classes are quite remarkable from other geometric 
points of view as well. In a subsequent paper
we will show that such representations can be characterized as 
cohomogeneity one actions on irreducible isoparametric submanifolds 
of the Euclidean space. Moreover, one can classify them and it turns out 
that, up to two exceptions, all such representations can be constructed 
as in the examples above.

The proof of \tref{verythird} (and of \tref{special}) follows the program 
outlined in the beginning of this subsection.
Analyzing the existence of boundary and reflections in the quotient $V/G^0$,  
we can exclude all but very few  representations, all of them
on small dimensional vector spaces.
It is a quite unfortunate but  unavoidable issue that one is left with 
finitely many cases that cannot be excluded directly.
Up to this point one could argue in any dimension $k$ of $G^0$, 
but in higher dimension some of these remaining 
possibilities will indeed give rise to true reductions, as it happens for~$k=7$.
In our case of~$k\leq 6$, the remaining irreducible representations of~$G^0$ 
all have cohomogeneities at most~$5$.
We rule them out by providing in the second part of the paper a complete 
classification of all
irreducible representations of cohomogeneity~$4$ or~$5$. 
This result is also used to exclude two cases
in the proof of \tref{special} that we could not exclude geometrically.    
We think that this classification is of independent interest and may  
be useful for other considerations as well.

\begin{thm}  \label{mainclass}
The non-polar irreducible representations of connected
compact Lie groups of cohomogeneities~$4$ or~$5$
as well as their 
copolarities and presence or not of boundary in the orbit space
are listed in the following two tables.
\[\begin{array}{|c|c|c|c|}
\hline
G & \rho & \textsl{Copolarity} & \textsl{Boundary} \\
\hline
 \SO3  & \R^7 & \textrm{trivial} & \textrm{no} \\
 \U2 & \C^4 & \textrm{trivial} & \textrm{no} \\
\SO3\times\G & \R^3\otimes_{\mathbf R}\R^7 & 2 & \textrm{yes} \\  
\SU3 & S^2\C^3 & 2 & \textrm{yes} \\
\SU6 & \Lambda^2\C^6 & 2 & \textrm{yes} \\ 
\SU3\times\SU3 & \C^3\otimes_{\mathbf C}\C^3 & 2 & \textrm{yes} \\
 \E6 & \C^{27} & 2 & \textrm{yes}\\
\hline  
\end{array} \]
\begin{center}
\textsc{Table 1: Cohomogeneity $4$}\label{cohom4}
\end{center}

\smallskip

\[\begin{array}{|c|c|c|c|}
\hline
G & \rho & \textsl{Copolarity} & \textsl{Boundary} \\
\hline
\SU2 & \C^4 & \textrm{trivial} & \textrm{no} \\
\SO3\times\U2  & \R^3\otimes_{\mathbf R}\R^4 & \textrm{trivial} & \textrm{yes}  \\
 \SU4 & S^2\C^4 & 3  & \textrm{yes} \\
\SU8 & \Lambda^2\C^8 & 3 & \textrm{yes} \\ 
\SU4\times\SU4 & \C^4\otimes_{\mathbf C}\C^4 & 3 & \textrm{yes} \\
\SO4\times\Spin7 & \R^4\otimes_{\mathbf R}\R^8 & 3 & \textrm{yes} \\
  \U3\times\SP2 & \C^3\otimes_{\mathbf C}\C^4 & 7 & \textrm{yes} \\
\hline  
\end{array} \]
\begin{center}
\textsc{Table 2: Cohomogeneity $5$}\label{cohom5}
\end{center}
\end{thm}

\subsection{The case $k=1$} We would like to explain the general strategy of 
this paper by sketching the proof
of \tref{verythird} in the case of $k=1$.  Thus let us assume  
that $\tau:G \to\OG{V}$ is the minimal reduction
of an irreducible representation $\rho:H \to \OG{W}$ of a 
connected group~$H$.  Assume further that $\dim (G)=1$, i.e.,
that $G^0$ is a circle $\U 1$.  Since $\rho$ is irreducible, so must be 
$\tau$ (\lref{folk}).  
Hence the action of $G$ must act transitively on the set of $G^0$-isotypical components of $V$. Since $G^0$ is a circle
and the action is effective, we deduce that there is only one 
$G^0$-isotypical component and that the representation of $G^0$ is given by the complex
multiplication of $\U 1$ on a complex vector space $V= \C ^l$.  
For $l=1$, the action is polar, thus it is not the minimal reduction. 
Hence $l\geq 2$, and $V/G^0$ is the cone over the complex projective space $\C P^{l-1}$.
Since it does not have boundary, \pref{veryfirst} and \pref{verysecond} imply that
some element in $G/G^0$ acts on $V/G^0$ as a reflection at a subspace of codimension $1$.
Thus we obtain a  reflection at a totally geodesic hypersurface of codimension $1$ in $\C P^{l-1}$.
But such hypersurfaces exist  only  for $l=2$.

\subsection{Questions}
We would like to finish the introduction by formulating some basic questions 
about quotient-equivalence classes, closely related to our results. 

\begin{quest}
Assume that $\rho_i :G_i \to \OG{V_i}$, where $i=1$, $2$,
are effective, reduced and quotient-equivalent.
Is it true that the $\rho _i$ must be orbit-equivalent?
\end{quest}

\begin{quest}
Let the irreducible representation of $G$ on $V$ be reduced. What is the isometry group of the quotient space $V/G$?
Can it be much larger than the $N(G) /G$, where $N(G)$ is the normalizer of 
$G$ in $\OG V$?
\end{quest}

\begin{quest}
For a representation $\rho :G \to \OG V$ of a group $G$ 
consider the following four conditions.
\begin{enumerate}
\item[(C1)] There is an orbit-equivalent action of some group $G'$ 
which has non-trivial principal isotropy groups.
\item[(C2)] The representation has non-trivial copolarity.
\item[(C3)] The action has a non-trivial reduction.
\item[(C4)] The quotient $V/G$ has non-empty boundary.
\end{enumerate}

We have implications (C1) $\Rightarrow$ (C2) $\Rightarrow$ (C3) 
$\Rightarrow$ (C4). We do not know a single representation which satisfies~(C3)
but not~(C1) and only very few representations of
connected groups satisfying~(C4) but 
not~(C1). Are there some  reverse  implications?
\end{quest}

For all representations appearing in our main results the conditions~(C1), 
(C2) and~(C3) are equivalent.

\begin{quest}
What are in  general the relations between cohomogeneity and copolarity?
\end{quest}

\begin{quest}
Is there a general description of all representations $\rho :G \to \OG {V}$ with trivial copolarity 
and non-empty boundary of  $V/G$?
\end{quest}

\subsection{Structure}  
The paper is divided into two parts. 
In the first part we study the geometry of reductions and prove
Theorems~\ref{verythird} and~\ref{special}. After a section on
preliminaries, we investigate orbifold parts of quotient
spaces in section~\ref{secreflections}. We use orbifold fundamental 
groups to reduce the investigation of quotient equivalent actions of 
$G_i$ on $M_i$, $i=1$, $2$, to the case where 
$G_i/G_i ^0$ acts on the subquotient $M_i/G_i ^0$ as a 
\emph{reflection group} (\pref{onereflection}). Moreover,
we find an easy  criterion when one can replace one of the groups 
$G_i$ by its identity component (\pref{liftnoboundary}).  As  special cases we deduce
\pref{verysecond} and \cref{principalcoxeter}. 
In section \ref{sectrivial} we study the restrictions imposed by the 
triviality of the principal isotropy groups on the strata of low 
codimensions. In section \ref{secbasic}, we specialize to representations
 and prove some basic structural results: the invariance of irreducibility 
(\lref{folk}) and  
\pref{veryfirst}. The results from these preparatory sections may be 
useful for other related problems as well.
 
In section \ref{secmain}, the technical heart of the paper, 
we study irreducible representations whose restricted representations 
to the identity component are reducible and prove \pref{mainproposition},
where our special representations of maximal tori $T^k$ of $\SU {k+1} $ 
turn up. In section \ref{secexample}
we apply results of sections~\ref{secreflections}, \ref{sectrivial} 
and~\ref{secbasic}, 
to study the geometry of $\C ^{k+1} /T^k$. In section \ref{secconclusion}, we collect the harvest
and prove \tref{special}.   In section \ref{secsetting}, we start with the proof of \tref{verythird},
which is finished in the two subsequent sections, dealing with the cases of connected, respectively, 
disconnected minimal reduction of our original presentation.

In the final part of the proof of \tref{special} as well as of \tref{verythird}, 
a few cases that we cannot resolve by direct geometric arguments
show up. 
To exclude these few a priori  possible reductions, we
rely on the classification of irreducible representations of low
cohomogeneity (\tref{mainclass}). This theorem is proved independently
in the second part of the paper, by invoking
classifications of representations of simple groups with low cohomogeneity, 
followed by a straightforward but a bit tedious analysis of tensor products 
of representations. In the last section we discuss copolarities of the found
representations.

We wish to thank the referee for his valuable comments and criticism which have
substantially helped improve this paper.

\bigskip

\begin{center}
\large\bf
Part 1. Geometry of reductions

\end{center}

\medskip

\section{Preliminaries}
In this section we collect a few basic results about 
actions of a compact group of isometries~$G$ 
on a connected complete Riemannian manifold~$M$. 
We assume that the actions are effective. 
In the later sections, $M$ will always be either an 
Euclidean space or an Euclidean sphere.   
Let $X$ be the quotient space $M/G$ with the induced quotient metric.
 
\subsection{Stratification}  \label{stratific}
For a point $p\in M$
we denote by $G_p$ its isotropy group,  
and we denote by $St(p)$ the stratum 
of $p$, i.e., the \emph{connected component} 
through $p$ of the set of points
$q\in M$ whose isotropy groups $G_q$ are conjugate to $G_p$. 

Denote by $x$ the image point $x=G\cdot p \in X$.  
The stratum $St(p)$ is a manifold and projects to a 
Riemannian totally geodesic submanifold $St_X (x)$ of the space $X$, 
called a stratum of $X$.
  
The dimension of the stratum $St(p)$ coincides with
the topological dimension of $\dim(G\cdot F)$, where $F$ 
is the connected 
component through $p$ of the set of fixed points  
of $G_p$. We have   $\dim(G\cdot F) = f_p+g-n_p$, where $f_p$ is the dimension 
of $F$, where  $g$ is the dimension of $G$ and $n_p$ is the dimension of the
normalizer $N(G_p)$ of $G_p$ in $G$.    

Locally at a point $p\in M$, the orbit decomposition of 
$M$ is completely determined by the \emph{slice representation}
of the isotropy group $G_p$ on the normal space 
$\mathcal H_p:=N_p(G\cdot p)$ to the orbit $G\cdot p$. 
According to our convention $St(p)$ is connected, therefore the 
equivalence class of the slice representation along $St(p)$ is constant. 
The set of fixed vectors of $G_p$ in $\mathcal H_p$
is tangent to the stratum $St(p)$,
and the action of $G_p$ on 
its orthogonal complement $\mathcal H_p^\dagger$ in $\mathcal H_p$ 
has cohomogeneity $\dim(\mathcal H_p^\dagger/G_p)$, which  
equals to the codimension of the stratum $St_X(x)$ in $X$,
where $x=G\cdot p$.
 
A point $p\in M$ is called \emph{regular} if 
$\mathcal H_p^\dagger$ is trivial. It is called \emph{exceptional} 
if it is not regular and the action of $G_p$ on 
$\mathcal H_p^\dagger$ has discrete orbits. If it is neither
regular nor exceptional, it is called \emph{singular}.
The set $M_{reg}$ of all regular points in $M$ is open and dense,
and $X_{reg}=M_{reg}/G$  is  
connected. $X_{reg}$ is the stratum corresponding to the unique 
conjugacy class of  minimal appearing isotropy groups. 
These isotropy groups are called the \emph{principal isotropy groups}.

\subsection{Boundary} \label{boundary}
The set $X_{reg}$ of regular points of $X$ is exactly the set of points that 
have neighborhoods isometric to Riemannian manifolds.  It is the unique 
maximally dimensional stratum in $X$. 
By definition, the \emph{boundary} of $X$ is the closure of union of all strata 
that have codimension~$1$ in~$X$.
It is denoted by $\partial X$.
A point $p\in M$ is mapped to a stratum of codimension $1$ in $X$ if 
and only if the isotropy group $G_p$ acts
on  $H_p^\dagger$  with cohomogeneity $1$. 
Since boundary points will play a special role in subsequent considerations, 
we are going to call a point $p\in M$ that projects to a 
point on a stratum of codimension $1$ 
in $X$ an \emph{important point}, or, if the action needs to be specified, a
$G$-\emph{important point}.

Let now $G'$ be a normal subgroup of $G$, with finite quotient 
$\Gamma =G/G'$. Then $\Gamma$ acts by isometries
on $X':= M/G'$, and $X=X'/\Gamma$. Since $X$ and $X'$ have the 
same dimension, and since strata of $X'$ are mapped to unions of strata of $X$, we have $\pi (\partial X')  \subset \partial X$. Here, $\pi :X'\to X$ is the canonical projection.  On the other hand, any $G$-important point $p$ 
that is not $G'$-important must be $G'$-regular.  
In this case, we deduce that $G_p /G'_p$ must act on the normal space 
$\mathcal H_p$ as a single reflection at  a 
hyperplane  in $\mathcal H_p$.
  
\subsection{Copolarity and polar actions} \label{copolar}
We refer to \cite{got} for  a detailed discussion of the following notions.
A \emph{generalized section} of  the action of $G$ on $M$ is a 
connected complete totally geodesic submanifold $\Sigma$ whose 
intersection $\Sigma\cap M_{reg}$ with the set of regular points is not 
empty and satisfies $\mathcal H_p\subset T_p \Sigma$ for any 
$p\in \Sigma\cap M_{reg}$. For any minimal generalized section $\Sigma$, the 
group $G_{\Sigma}=\{g \in G | g \Sigma= \Sigma \}$ acts on $\Sigma$ 
with a kernel which we quotient out to get the effectivization
$\bar G_\Sigma$, and then the canonical map $\Sigma/ \bar G_{\Sigma}  \to M/ G$
is an isometry. The minimal dimension of such $\bar G_{\Sigma}$ is called 
the \emph{copolarity} of the action. In addition,
we say that the copolarity  of  the action of $G$ on $M$  is  non-trivial if $\Sigma \neq M$, i.e., if~$G_{\Sigma} \neq G$.
 
If the action of $G$ on $M$ has non-trivial principal isotropy groups then 
a connected component containing regular points
of the set of fixed points of any principal isotropy group is a generalized 
section. Thus such an action has non-trivial copolarity.
  
An
action of another compact Lie group
$G'$ by isometries on another connected complete
Riemannian manifold $M'$ is called \emph{orbit-equivalent}
to the action of $G$ on $M$ if both actions have the same orbits, 
up to isometry. More precisely, they are orbit-equivalent if there 
exists an isometry $F:M\to M'$ such that $F(G(p))=G'(F(p))$ for every
$p\in M$. In this case, we can identify $M$ and $M'$ via $F$ and
view $G$ and $G'$ as subgroups of the isometry group of $M$. 
If $G'$ does not coincide with $G$ inside the isometry group of $M$ 
then the group generated by $G$ and $G'$ is also
orbit equivalent to both actions and has non-trivial principal isotropy 
groups. Thus this group and therefore the actions of $G$ and $G'$ 
have non-trivial copolarity in this case.

An action is called \emph{polar} if it has copolarity $0$, i.e., 
if and only if it admits a generalized section $\Sigma$
with $\dim (\Sigma )= \dim M/G$ 
(in which case $\Sigma $ is called a \emph{section} of the action). 

If $M$ is a Euclidean space $V$, then by the above, 
any non-trivial generalized section defines a reduction.  
We define the \emph{abstract copolarity}
of a representation $\rho :G \to \OG {V}$ to be the dimension  
of the underlying group in a minimal reduction of $\rho$. 
Thus the abstract copolarity  is bounded above by the copolarity.

It is known that an orthogonal representation  $G$ on $V$ is polar 
if and  only if it is orbit equivalent to 
the isotropy representation of a symmetric space~\cite{dad}. 
Another equivalent formulation is that the set of regular points 
$X_{reg}$ of $X$ is flat~\cite{hlo,ale}. 
Therefore, if a representation has a reduction to a discrete group, 
it must be polar. Thus a representation has abstract copolarity 
$0$ if and only if it has copolarity $0$. 

The classification result of \cite{str} implies, in terms of copolarity, that 
any representation of cohomogeneity at most $3$ has copolarity at most $1$.

\begin{rem}  
 The statement of \tref{special}  implies that for all representations satisfying the conditions
of \tref{special}, the copolarity and abstract copolarity coincide.  Due to \tref{verythird}, the equality of copolarity and 
abstract copolarity is also true  for all irreducible 
representations of connected groups if the abstract copolarity is at most $6$.  Looking into the last section of this paper,
one can also deduce that this equality holds true for all representations 
described in  \tref{mainclass}.
\end{rem}

\section{Reflections in quotients} \label{secreflections}
\subsection{Formulation}
Given two isometric quotients~$M_1/G_1$ and~$M_2/G_2$ of 
Riemannian manifolds modulo compact groups of isometries, one  
would like to replace the groups by some smaller subgroups of 
finite index (for instance, by the identity components) preserving 
the property of having isometric quotients.
In this section we are going to prove two  useful criteria for such 
reductions. The first one is quite general:

\begin{prop} \label{liftnoboundary}
Let $M_1$, $M_2$ be simply connected Riemannian manifolds. 
Let $G_i$ be a compact group of isometries of 
$M_i$, for $i=1, 2$. Assume that $M_1/G_1$ and $M_2/G_2$ are isometric. 
If the quotient $M_1/G_1 ^0$ has no boundary,
then for any subgroup $G_2'$ of finite index in $G_2$ there 
is a subgroup $G_1'$ of finite index in $G_1$
such that $M_2/G_2' =M_1/G_1'$.
\end{prop}

To formulate our criterion in the presence of boundaries  
we need a definition. A \emph{reflection} on a Riemannian manifold 
is an isometry  whose set of fixed points
has  codimension $1$.  Let $X=M/G$ be a quotient space. 
An isometry of $X$ is called a \emph{reflection} on $X$ if 
its restriction to the regular part $X_{reg}$
is a reflection. A discrete group of isometries of $X$
that is generated by reflections is called a \emph{reflection group} 
on $X$.  We refer to \cite{AKLM}
for more about reflection groups on manifolds.

The second result of this section is the following:
\begin{prop}  \label{onereflection}
Let $G_i$ be a compact group of isometries of a simply connected 
Riemannian manifold $M_i$ for $i=1,2$.
Assume that $M_1/G_1$ and $M_2/G_2$ are isometric. 
Then there are normal subgroups of finite index $G_i^+$ in $G_i$ 
such that the group $G_i^+/G_i^0$  acts on the quotient $M_i/G_i^0$ as a 
reflection group, $i=1$, $2$,  
and $M_1/G_1 ^+$ and $M_2/G_2 ^+$ are isometric. Moreover, the image of $G_i ^+ /G_i ^0$ in the 
isometry group of $M_i /G_i ^0$ is a Coxeter group. 
In particular if, say, $G_1$ is connected,
we can take $G_i^+=G_i$ for $i=1$, $2$ and then $G_2/G_2^0$ acts on
$M_2/G_2^0$ as a reflection group. 
\end{prop}

To prove the results we are going to study more closely a 
slightly larger part of the quotient $X$ than $X_{reg}$, 
its \emph{orbifold part} $X_{orb}$. 
  
\subsection{Riemannian orbifolds and their fundamental groups}\label{fundgroups}

We are going to use a bit about orbifolds and orbifold fundamental groups.  
We refer the reader to the notes by Thurston~\cite{thu} 
and to those by Davis~\cite{Davis}.

A \emph{Riemannian orbifold} is a metric space $C$ where each point has a neighborhood isometric to a finite quotient of a 
smooth Riemannian manifold. Since it is locally represented as  a quotient in a unique manner, it comes along with a natural
stratification and a unique underlying structure of a smooth orbifold.

Let $C$ be a connected Riemannian orbifold.
Let $\pi _1 ^{orb} (C)$ denote the \emph{orbifold fundamental group} 
of $C$ (cf. \cite{Davis}).
This group acts as a group of discrete isometries on a 
connected Riemannian orbifold $\tilde C$,
called the \emph{universal orbifold covering} of $C$, 
such that $\tilde C / \pi _1 ^{orb} (C) =C$.
For any other presentation of the orbifold $C$ as a quotient $C=B/\Gamma$ 
of a connected Riemannian orbifold
$B$ modulo a discrete group of isometries $\Gamma$, there is an 
(essentially unique) normal subgroup $\Gamma '$
of $\pi _1^{orb}(C)$, such that $\tilde C /\Gamma ' =B$ 
and such that $\Gamma = \pi _1 ^{orb}(C) /\Gamma '$.

A Riemannian orbifold $C$ has a trivial orbifold fundamental group 
if and only if it coincides with its universal orbifold covering 
$\tilde C$. In this case $C$ has no boundary and it is simply 
connected as a topological space. 
In general, one can write a presentation of the orbifold 
fundamental group of $C$
in terms of its usual fundamental group and its strata of 
codimension $1$ and $2$~\cite{Davis}.
Looking at  this presentation of the orbifold fundamental group one 
notes that the embedding $C\setminus \partial C \to C$ induces an injection $\pi _1 ^{orb} (C\setminus \partial C) \to 
\pi _1 ^{orb} (C)$.

If $\pi _1 ^{orb} (C \setminus \partial C) =1$  
(essentially, the only case that will be of importance in the 
sequel, see  \lref{simplycon}) then the presentation 
of~$\pi _1 ^{orb}(C)$ has the following simple form.
For any stratum  $\Sigma$ of codimension one in $C$ one takes a 
generator $w_{\Sigma}$ that has by definition order $2$.  
Whenever codimension one strata $\Sigma _1$, $\Sigma _2$ 
meet at a stratum $P$ of codimension $2$,
(i.e.~$P \subset \bar \Sigma _1 \cap \bar \Sigma _2$) one adds a relation 
$(w_{\Sigma _1} \cdot w_{\Sigma _2}) ^ {m(P)} =1$, where $m(P)$ is 
the natural number defined such $\Sigma _1 $
and $\Sigma _2$ meet at $P$ at the angle $\pi /m(P)$.  
In particular, $\pi _1 ^{orb} (C)$ is a Coxeter group.
If all occurring angles $\pi/ m(P)$ are equal to~$\pi /2$, 
then the group $\pi _1 ^{orb}(C)$ must be Abelian or infinite.



\subsection{Reflections in orbifolds}
Let $C$ be a connected Riemannian orbifold.
For a discrete group $\Gamma$ of isometries of $C$ we denote by 
$\Gamma _{refl}$
the subgroup of $\Gamma$ that is generated by reflections of $C$ that 
are contained in
$\Gamma$.  Since a conjugate of a reflection is a reflection, 
$\Gamma _{refl}$
is a normal subgroup of $\Gamma$.

The following lemma will provide us with sufficiently many reflections:

\begin{lem} \label{keyreflection}
Let $B$ be a Riemannian orbifold. 
Let $\Gamma$ be a discrete group of isometries of
$B$. Let $C$ be the orbifold $C=B/\Gamma$.  Set $C_0=C\setminus \partial C$. 
If the orbifold fundamental group of $C_0$ is trivial, then  
$\Gamma$ is a reflection group on $B$. 
\end{lem}

\begin{proof}
Denote by $\Gamma '$ the quotient group $\Gamma '= \Gamma / \Gamma _{refl}$.
Then $\Gamma '$ acts by isometries on $B' :=B/\Gamma_{refl}$ 
such that $B'/\Gamma ' =C$.
If an element $w$ in $\Gamma '$ acts as a reflection on $B'$, 
then we can find a manifold point $p$ in $B$ that is projected 
to a manifold point in $B'$,  but  whose projection to $C$ lies on  
 a stratum of codimension $1$ in $C$. Thus this point $p$ must be fixed 
by a reflection in $\Gamma$ 
that is not in $\Gamma _{refl}$, providing a contradiction.

Hence $\Gamma '$ does not contain reflections of $B'$.  Therefore, the 
projection $B' \to C$  has the property that the preimage of a boundary 
point in $C$ is a boundary point in $B'$. Hence the preimage
of $C_0$ is exactly $B'_0 = B'\setminus \partial (B')$, i.e., a connected 
orbifold.
Hence $C_0 = B ' _0 / \Gamma '$. 
Since $\pi _1 ^{orb} (C_0)=1$,  the group  $\Gamma '$
acts trivially on $B'$. Thus $\Gamma = \Gamma _{refl}$.
\end{proof}

Consider now the action of $\pi _1 ^{orb} (C)$ on the universal covering 
$\tilde C$ of $C$.
We obtain the reflection group $\pi_1 ^{orb}(C)_{refl}=:\pi_{1,refl}^{orb}(C)$
of isometries of $\tilde C$.   
The quotient $\tilde C / \pi _{1,refl}^{orb}(C)$  
is a Riemannian orbifold, which will be denoted by 
$C_{refl}$. The quotient group $\pi_{1,nonrefl}^{orb}(C) : 
=\pi _1 ^{orb}(C) /\pi _{1,refl}^{orb}(C)$ acts on $C_{refl}$
with $C_{refl} /\pi _{1,nonrefl}^{orb}(C) = C$. By construction, $C_{refl}$ 
is the unique minimal orbifold covering of $C$
with property that its orbifold fundamental group is generated by 
reflections.

The following observation is  probably well known to the experts.
We include it here with a sketchy proof for the sake of completeness:

\begin{prop} \label{sketchy}
In the notations above we have 
$\pi _{1, nonrefl} ^{orb}(C) = \pi _1 ^{orb} (C \setminus  \partial C)$.  
Thus for any Riemannian orbifold there is a split exact sequence:
$$1\to \pi _{1, refl} ^{orb} (C) \to  \pi _1 ^{orb}  (C) \to \pi _1 ^{orb} (C\setminus \partial C) \to 1$$
\end{prop}

\begin{proof}
The proof of \lref{keyreflection} shows that $\pi _{1, nonrefl} ^{orb} (C)$ 
does not contain reflections
of $C_{refl}$. Moreover, for the projection from $C_{refl}$ to $C$, the 
preimage of $C\setminus \partial C$
is the connected Riemannian orbifold $C_{refl} \setminus \partial C_{refl}$.  

Now, by definition, the orbifold fundamental group of $C_{refl}$ is 
generated by reflections. Looking at the presentation of the 
orbifold fundamental group in terms of the strata~\cite{Davis}, 
one deduces that $C_{refl} \setminus \partial C_{refl}$ has trivial orbifold fundamental group.
Hence $\pi _{1,nonrefl}  ^{orb}(C)$ acts on the simply connected orbifold $C_{refl} \setminus \partial C_{refl}$
with quotient orbifold identified with $C\setminus \partial C$. This implies the result.
\end{proof}

The following result has been proven for Riemannian manifolds in~\cite{AKLM}.  
Instead of the proof given below, 
one could obtain the proof for orbifolds along the same lines as in there.

\begin{lem} \label{lemcoxeter}
Let $B$ be a Riemannian orbifold with trivial orbifold fundamental group.  Let $\Gamma$ be a reflection group
on $B$. Then $\Gamma$ is a Coxeter group.
\end{lem}

\begin{proof}
 Consider the quotient $C= B/\Gamma$. Since $\pi _1 ^{orb} (B)=1$,  the group $\Gamma$ is the orbifold fundamental 
group of $C$.  Since $\Gamma$ is a reflection group, we deduce from \pref{sketchy}, that $\pi_1 ^{orb} (C\setminus \partial C) =1$.  The discussion at the end of subsection \ref{fundgroups} shows that $\Gamma$ is a  Coxeter group.
\end{proof}

\subsection{Orbifold points in quotients}
We call a point $x$ in our quotient $X=M/G$ an orbifold point if it has a 
neighborhood isometric to a Riemannian orbifold.
It has been shown in \cite{LT} that a point $p\in M$ is projected to an 
orbifold point in $X$ if and only if the
slice representation of $G_p$ on $\mathcal H_p$ is polar. 
The orbifold $X_{orb}$ of all orbifold points in $X$ 
is open, connected and it is a union of strata.
It contains all strata that have codimension at most $2$ in $X$, 
in particular, all $G$-important orbits. 
The set $X_{orb}$  has a non-empty boundary if and only if $X$ 
has non-empty boundary.

The following result, probably folklore, can be found in \cite{Lyt}:
 
\begin{lem} \label{simplycon}
Let $M$ be a simply connected  
complete Riemannian  manifold. 
Let $G$ be a connected compact group of isometries of $M$. Let $X$ be 
the quotient $M/G$.  Let $X_{orb}$ be the set of orbifold points 
in $X$ and set $X_0 = X_{orb} \setminus \partial X_{orb}$. Then
$X_0$ is exactly the set of non-singular $G$-orbits. Moreover, 
$X_0$ has trivial orbifold 
fundamental group. 
 \end{lem}

In particular, the above lemma says, that if $M$ is simply connected and 
$G$ connected, no $G$-important point may lie on an  exceptional orbit.

The open subsets $X_{reg}$ and $X_{orb}$ are dense and convex in the 
quotient space $X$. Thus for a pair of quotients $X=M/G$ and $Y=N/H$  
any isometry between $X_{orb}$ and $Y_{orb}$  (or between~$X_{reg}$ 
and~$Y_{reg}$) extends uniquely to an isometry between their completions~$X$ 
and~$Y$.

Consider again the quotient $X=M/G$.
If $G'$ is a subgroup of finite index in $G$ then $\pi (X' _{orb} ) =X_{orb}$,
where $X' =M/G'$ and $\pi : X' \to X$ is the canonical projection.

Let $X^0 $ be the subquotient $M/G^0$. Consider the orbifold
parts $X_{orb}$ and $X_{orb} ^0$. 
The group $G$ acts on $X^0 _{orb}$ by isometries
and we obtain a homomorphism $j$ from $G/G^0$ onto a finite group $D$ of isometries of $X^0_{orb}$, such 
that $X_{orb} = X_{orb} ^0 /D$.  Thus for the orbifold fundamental groups 
$\Gamma = \pi _1 ^{orb} (X _{orb})$ and 
$\Gamma _0 =  \pi _1 ^{orb} (X^0 _{orb})$, the group $\Gamma_0$
is a normal subgroup of $\Gamma$ with $D=\Gamma /\Gamma _0$.  

Any subgroup $G'$ of finite index in  $G$ has a subquotient $X' =M/G'$
such that the  orbifold fundamental group of $X' _{orb}$ is contained 
in $ \pi _1 ^{orb} (X _{orb})$   and contains 
 $ \pi _1 ^{orb} (X ^0 _{orb})$. On the other hand, any group 
$\Gamma '$ with  
$\Gamma _0 \subset \Gamma ' \subset \Gamma $  projects to a subgroup 
of $D$. Taking the preimage of this subgroup  under $j$ in $G$ we obtain  
a subgroup  $G'$ of finite index in $G$. Then  the orbifold part 
$X' _{orb}$ of  the quotient $X' =M/G'$  has $\Gamma '$ as its orbifold 
fundamental group.

Now we are in  position to prove the main results of this section.

\begin{proof} [Proof of \pref{liftnoboundary}]
Set $X=M_1 /G_1 =M_2 /G_2$. By assumption, $X^0 =M_1 /G_1 ^0$ has no 
boundary.  Hence
$X^0 _{orb}$ has trivial orbifold fundamental group by  
\lref{simplycon}.   Consider the orbifold covering
$X' _{orb}  =(M_2 /G_2') _{orb}$
of $X_{orb}$ defined by the subgroup $G_2 '$ of $G_2$.  
Since $\pi _1 ^{orb}  (X^0 _{orb})=1$, using considerations 
preceding the proof,
we find a subgroup $G_1 '$ of $G_1$ such that $(M_1 /G_1 ')_{orb} =X' _{orb}$. 
\end{proof}

\begin{proof}  [Proof of \pref{onereflection}]
Let $X =M_1 /G_1=M_2/G_2$ and consider the orbifolds 
$(M_i/G_i^0)_{orb}$ for $i=1$, $2$. We deduce from
\lref{simplycon} and \lref{keyreflection}  that  $\Gamma ^i =\pi _1 ^{orb}
((M_i/G_i ^0) _{orb})$ is generated by reflections.  
Consider now the subgroup $\Gamma  _{refl}$ of
$\Gamma = \pi _1 ^{orb} (X_{orb})$ that is generated by all 
reflections in $\Gamma$. Then $\Gamma_{refl}$
contains $\Gamma ^i$, hence due to the considerations above, 
we find subgroups of finite index $G_i^+$ in $G_i$ 
such that $(M_i /G_i ^+) _{orb}$  is isometric to  
$(X_{orb})_{refl}$, i.e., to the quotient
of the universal orbifold covering of $X_{orb}$ modulo 
$\Gamma _{refl}$.  Hence $M_i /G_i ^+$ for $i=1$, $2$ are isometric.

Since $\Gamma _{refl}$ is generated by reflections on the 
universal orbifold covering of $X_{orb}$,
the group $\Gamma _{refl} /\Gamma ^i$ acts as a reflection 
group on $(M_i /G_i ^0)_{orb}$.	
By construction,  $\Gamma _{refl} /\Gamma ^i$ is exactly the image of
$G_i ^+ /G_i ^0$  in the isometry group of $M_i /G_i ^0$.

It remains to prove that $\Gamma _{refl} /\Gamma ^i$  is a Coxeter group.
However, we have seen this group  acts as a reflection group on $M_i /G_i ^0$ hence also on 
the Riemannian orbifold
$(M_i /G_i ^0)_{orb} \setminus \partial ((M_i /G_i ^0 ) _{orb})$, which has trivial orbifold 
fundamental group.
Now we use \lref{lemcoxeter} to deduce the claim. 
\end{proof}

Specializing to the case in which the $M_i$ are Euclidean spaces, we 
obtain \pref{verysecond} as a special case of Proposition~\ref{onereflection}.
Moreover, under the assumptions of this proposition,
if $G_2$ acts with trivial principal isotropy groups, then $G_2 /G_2 ^0$ 
embeds into the isometry group
of $M_2 /G_2^0$. If $G_1$ is connected, we deduce that $G_2 /G_2 ^0$ must 
be a Coxeter group. This proves \cref{principalcoxeter}.

 \section{Triviality of the principal isotropy group} \label{sectrivial}
\subsection{Boundary points}\label{bdpts}
Let the compact group $G$ act (effectively)  on the simply connected 
manifold $M$ with trivial principal isotropy groups.
Then the principal isotropy group of any 
slice representation is trivial as well.  If $p\in M$ is a 
$G$-important point then the isotropy group $G_p$ must
act transitively on the unit  sphere $S^a$ in $\mathcal H_p^\dagger$, 
the normal space to the stratum $St(p)$.  Since the action must have 
trivial principal isotropy as well, $G_p$ must be diffeomorphic 
to~$S^a$. But this can happen only for~$a=0$, $1$ or $3$.  
Note that if $a=1$ or $3$, 
then $G_p$ is contained in the identity component $G^0$ of 
$G$ and $p$ is $G^0$-important.
On the other hand, if $p$ is a $G^0$-important point, then it cannot 
lie on an $G^0$-exceptional orbit, due to 
\lref{simplycon}.  Thus $G_p$ must be non-discrete and $a\neq 0$ in this case.

Hence $p$ is $G$-important and not $G^0$-important 
if and only if $G_p$ is a group with only one non-trivial element~$w$. 
This element $w$ is an involution in $G\setminus G^0$ which normalizes $G^0$.  
Moreover, it acts as a reflection on~$M/G^0$.

Summarizing and using Subsection~\ref{stratific} we arrive at:
\begin{lem} \label{boundreduced}
Let the compact group $G$ with identity component $G^0$ act 
effectively on a simply connected manifold $M$. 
Assume that the principal isotropy group of $G$ is trivial. 
If a point $p$ is    $G$-important, then $G_p= S^a$, for $a$ equal 
to $0$, $1$ or $3$. We have $a=0$ if and only if $p$ is 
not  $G^0$-important. Moreover the stratum 
$St(p)$  through $p$ has  dimension $\dim St(p)=\dim (M) -a -1 = f_p+g-n_p$, 
where $f_p$ denotes the dimension of the connected component 
of the set of fixed points of $G_p$,
$g$ is the dimension of $G$ and $n_p$ is the dimension of the 
normalizer of $G_p$ in $G$.
\end{lem}

\subsection{Codimension two strata}
Let a \emph{connected} compact group $G$ act on a simply connected 
manifold $M$ with trivial principal isotropy
group. Let $X$ be the quotient $M/G$, 
and let $x$ be a point in a stratum of codimension two in $X$. 
Then $x\in X_{orb}$.  If $x$ is not contained in the boundary 
$\partial X$ it must be an exceptional orbit
of the $G$-action (Lemma~\ref{simplycon}). 
Otherwise, $x\in\partial X$. Take a point $p$ in the 
$G$-orbit corresponding to $x$.
Then the tangent cone $C_x X$ at $x$ to $X$ is 
isometric to the quotient $\mathcal H _p /G_p$
of the slice representation.   
The assumption that $x$ lies in a stratum of codimension two 
implies that $G_p$ fixes a subspace $\R ^{k-2}$ of 
$\mathcal H_p$ and acts on the orthogonal complement
$\mathcal H_p^\dagger$ with codimension two.

The representation of $(G_p)^0$  on $\mathcal H_p^\dagger$ 
is of cohomogeneity two 
and the quotient $\mathcal H_p^\dagger / (G_p) ^0$ is isometric to 
$\R ^2 /D_m$,  where $D_m$
is the dihedral group of order $2m$ that is generated by two reflections 
at lines enclosing the angle 
$\pi /m$.  The classification of representations of cohomogeneity 
two~\cite{hl} implies
that $m$ can be only~$2$, $3$, $4$ or~$6$  
(much more sophisticated topological argument is given in 
\cite{Mu2}). Moreover, the classification shows that the action of 
$(G_p) ^0$ on $\mathcal H_p^\dagger$
has non-trivial principal isotropy groups, 
possibly unless if~$m=2$. 
If  the action of $(G_p) ^0$ has non-trivial principal
isotropy, so does the action of $G$.    
Thus $\mathcal H _p / (G_p) ^0$ is isometric to
 $\R ^{k-2} \times \R ^2 /D_2$.

We claim that the action of $G_p$ on $\mathcal H _p$ is orbit-equivalent to the action of $(G_p) ^0$. 
Otherwise, $G_p /(G_p)^0$  
acts as a non-trivial group of isometries on $\mathcal H_p^\dagger / (G_p) ^0$.
However, the only non-trivial isometry of the quadrant $\R ^2 /D_2$ is the reflection at the midline of the quadrant.
Any non-zero  small vector  $v\in \mathcal H_p$ in the preimage of this line is $G_p$-important and $(G_p) ^0$-regular.
Exponentiating $v$, we find a point $p'$ in $M$ close to $p$ that is $G$-important and lies on an exceptional orbit.
This contradicts   \lref{simplycon}.

Thus, we have shown   that  the tangent cone $C_x X$ is 
isometric to $\R ^{k-2} \times \R ^2 /D_2$.  
But this means that the two strata 
of codimension $1$, whose closures
contain the point $x$, meet at the point $x$ at a right angle. 

\begin{lem} \label{isotropy}
Let a connected compact Lie group $G$ act isometrically on a simply connected 
manifold $M$. Assume that the principal isotropy groups of the action 
are trivial.  Then codimension $1$ strata of $X=M/G$ can only meet at 
a right angle, i.e., for  any point $x$ in a
stratum of codimension two in $X$ which is contained in the boundary,
the tangent cone $C_x X$ at $x$ is isometric
to $\R ^{k-2}  \times \R ^2 /D_2$.  Here, $D_2$ is the dihedral group of order $4$ generated by reflections at two orthogonal lines. 
\end{lem}

\subsection{Nice involutions}  \label{newsubsec}
Assume now that a (possibly disconnected) 
group $G$ acts on Riemannian manifold $M$ with 
trivial principal isotropy group.
Assume moreover that $\Gamma=G/G^0$ acts on $M/G^0$ as a reflection group.  
Since the action has trivial principal isotropy groups, 
the action of $\Gamma$ on 
$M/G^0$ is effective.  By definition, the group $\Gamma$ 
is generated by elements  $w'\in G$ that act on
the quotient $M/G^0$ as reflections. Given any $w'\in G$ that acts 
on $M/G^0$ as a reflection, we take a regular point $x\in M/G^0$ 
that is fixed by $w'$. Let $p$ be a preimage of 
$x$ in $M$. Then $p$ is $G^0$-regular and $G$-important.  
Thus $G_p$ contains only one non-trivial element $w$, equal
to $w'g$, for some $g\in G^0$. This element $w$ is an involution and
the connected component~$F$ through $p$ of its set of fixed points 
satisfies
$\dim (G\cdot F)= \dim (F) +\dim(G)- \dim (C)  = \dim (M) -1$, where $C$ is the centralizer of $w$ (hence the normalizer of $G_p$) in~$G$.
Since $w$ and $w'$ are equivalent modulo $G^0$, involutions of the kind 
of $w$ generate $\Gamma$; we will
call them \emph{nice involutions}.

\section{Basic observations} \label{secbasic}
In this section we collect some basic observations about  
quotient-equivalences. 
For a Euclidean space $U$, we will denote by $S(U)$ the unit sphere of $U$.

\subsection{Fixed points and origins}
Let $U$ be a Euclidean space and let $K$ be a nontrivial
closed subgroup of $O(U)$. 
Let $F$ be the set of fixed points
of $K$, and let $F^{\perp}$ be its orthogonal complement. Then $X=U/K$   
splits as $F\times (F^{\perp} /K)$.
Moreover, for any unit vector $v$ in $U \setminus F$  there 
is no unit vector  $v'$~in $ U$   with
$d(K\cdot v ,v')= 2$. Thus there is no geodesic in 
$U /K$ that starts in $K\cdot v$ and has the origin $K\cdot0=0_U$ as 
its midpoint.  
Thus $F$ contains the set of all lines through the origin. In particular, 
it is the unique maximal Euclidean factor of $U/K$.  

On the other hand, in the factor $F^{\perp} /K$  the origin $0_U$ 
is the only point that is not the midpoint of some geodesic 
(all other points lie on the ray that starts at $0_U$).

Therefore, for any other representation of a group $K'$ on an 
Euclidean space $U'$, any isometry $I:U/K \to U' /K'$
must be given as a product of isometries $I_1:F\to F'$ and
$I_2: F^{\perp}/K \to (F')  ^{\perp} /K'$.  
Moreover, the second isometry  $I_2$ must send the origin to the origin.

Changing our isometry $I$ by an isometry of the Euclidean space $F$, 
if needed, we may therefore assume
that $I$  sends the origin $0_U$ to the origin $0_{U'}$.  
From now on we will make this assumption.

Since the quotient of the unit  spheres $S(U) /K$ is just the unit 
distance sphere in the quotient, our isometry $I$ induces an isometry 
between the spherical quotients $I:S(U) /K \to S(U') /K'$.

If the set of fixed points $F$ is non-trivial, then $S(U)/K$ 
has diameter $\pi$.  On the other hand, assume
that the diameter of $S(U) /K$ is larger than $\pi /2$. 
Then, for some orbit $K\cdot v$ of a unit vector $v$,
the set of points in the unit sphere with distance $\geq \pi/ 2 +\epsilon$ to this orbit   
  is non-empty, for some positive $\epsilon$.
But this set is compact, convex, $K$-invariant and does not contain  
great circles. Hence
it has a unique center which must be fixed by $K$.

Thus we see that the action has non-zero fixed points if and only if 
the diameter of $S(U)/K$ is 
larger than $\pi /2$ (in which case it is equal to $\pi$).

\subsection{Invariance of reducibility} We want to prove that invariant 
subspaces can be recognized metrically and thus are invariants of the 
quotient-equivalence classes.

Thus let $K$ act on $U$ as above. Consider the restricted action on the 
sphere $S(U)$ with quotient $X=S(U) /K$. 
Assume that there are no  fixed points of $K$ in $S(U)$.

We claim that a closed
subset $Z\subset X$  has the form $S(V)/K$  for some 
$K$-invariant subspace $V$ if and only if there is  some $Z' \subset X$ 
such that $Z$ is the set of all points $z \in X$ 
with  $d(z,z') =  \pi /2$,
for all  $z'\in Z'$.

Namely, if $Z=S(V)/K$ where $V$ is $K$-invariant,  
one can consider its orthogonal complement  $V^{\perp}$
and set $Z' := S(V^{\perp})/K$. On the other hand, 
if $Z\subset X$ is given in terms of $Z'$ as above,
then the preimage of $Z$ in $S(U)$ is a compact
convex $K$-invariant subset  
of~$S(U)$. Since~$K$ does not have fixed points 
in $U$, this subset  must be a great subsphere of $S(U)$.  
Thus it is the unit sphere of a $K$-invariant subspace.  

This provides a metric description of projections of invariant 
subspaces under the assumption of the absence of fixed points. 
Combining it with the previous subsection we arrive at:

\begin{lem} \label{folk}
Let $\rho :K \to\OG{U}$ and $\rho' : K'\to\OG{U'}$ be quotient-equivalent 
representations, with projections $\pi:U\to U/K$ and $\pi' :U' \to U' /K'$.  
Then $\rho$ is irreducible 
if and only if $\rho '$ is. More precisely, if $I:U/K \to U'/K'$ is an 
(origin preserving)  isometry then
for any $K$-invariant subspace $V$ of $U$ the subset  
$\pi '^{-1} (I (\pi (V)))$ is a $K'$-invariant subspace
of $U'$.  
\end{lem}

\subsection{Existence of  boundary points}
Given a representation of $\rho :G \to \OG{V}$ of a group $G$, 
the copolarity of $G$ and of its identity component $G^0$
coincide \cite{got}. On the other hand, the corresponding statement about the
abstract copolarity is not clear to us.
The following lemma, which is most basic to our 
results, is formulated in terms of the identity
component. In view of this lemma, it seems that the 
abstract copolarity of the representation may always coincide with that of
the identity component. 

\begin{prop} \label{boundaryexist}
Let $\rho:G \to\OG{V}$ be an effective representation and let 
$G^0$ be the identity component of $G$.  If $X_0=V/G^0$ has an empty 
boundary then the representation $\rho$ is reduced.
\end{prop}

\begin{proof}
Assume that the boundary of $X_0$ is empty and assume that $X=V/G$ is 
isometric to $W/H$, for a representation of a group $H$ on $W$ with  
$\dim (W) < \dim (V)$. 

The ``unit sphere'' $Y_0$ in $X_0$, i.e., the distance sphere 
to the orbit of the origin, is isometric to $S(V) /G^0$. Since 
$X_0$ is the Euclidean cone over $Y_0$, the quotient $Y_0$ does not 
have boundary as well.   
The unit sphere $Y$ of $X$ is isometric to $S(V)/G$ and to $S(W)/H$. 
By construction, $Y$ is a finite quotient 
$Y=Y_0/\Gamma$, for $\Gamma=G/G^0$.

Since $Y_0$ has no strata of codimension $1$, 
we find an infinite geodesic in $Y_0$ that is contained 
in the set of $G^0$-regular orbits and starts at a point 
$y \in Y_0$, 
such that $y$ is projected to a regular point in $Y$.
Take a part $\gamma$ of this geodesic that has length $\pi$. 
Let $m$ be the index of this geodesic 
(i.e., the number of conjugate points along $\gamma$, counted with 
multiplicities). 
For the Riemannian submersion $S(V)_{reg} \to (Y_0)_{reg}$, 
consider any horizontal lift
$\eta $ of $\gamma$. Then the index $m$ of $\gamma$ is equal 
to the $L$-index of $\eta$ (i.e.~the number of $L$-focal
points along $\eta$), where $L$ is the $G^0$-orbit corresponding to the 
point $y$, and where $\eta$ is considered as an $L$-geodesic. But in the round sphere $S(V)$, the $L$-index
of any $L$-geodesic $\eta$ of length $\pi$  is exactly the dimension of $L$.
Thus $m=\dim (L)$ in this case.

Consider now the image $\gamma '$ of $\gamma$ in $Y$. It is contained 
in the orbifold part of $Y$ and is (by definition) an \emph{orbifold geodesic} 
that starts at a regular point. Consider a lift $\eta'$ of $\gamma '$  
to an $H$-horizontal geodesic in $S(W)$ that starts on a regular $H$-orbit 
$L'$.   It has been 
shown in \cite{LT} that the $L'$-index of the geodesic $\eta'$  is equal 
to the sum of the index of the orbifold-geodesic $\gamma '$
and a ``vertical index'',  a non-negative number that counts the number 
of intersection of $\eta '$ with $H$-singular orbits.
In particular, it is not smaller than $m$, the index of the 
orbifold-geodesic $\gamma'$.
Using again that the $L'$-index of $\eta'$ is given by the 
dimension of $L'$, we get
$\dim (L') \geq m$.
 
But this contradicts 
$\dim (V)= \dim (X) +\dim (L)> \dim (W) = \dim (X) +\dim (L') $. 
\end{proof}
 
\medskip

This result immediately implies \pref{veryfirst}.

\begin{rem}
Applying Wilking's transversal Jacobi 
equation~\cite[Cor.~10]{Wilk}, one can deduce in a similar way the 
following related statement.
Let a compact group $G$ act on a compact positively curved manifold $M$.
If the quotient $M/G^0$ has no boundary then the   action has trivial 
copolarity.
\end{rem}

\section{Main argument} \label{secmain} 

In this section we are going to prove the following result:

\begin{prop} \label{mainproposition}
Let a non-discrete  group
$G$ act with trivial copolarity on a Euclidean vector space $V$.
Assume that $G/G^0$ acts on $V/G^0$ as a reflection group. Assume also
that~$G$ acts irreducibly but that the action
of $G^0$ is reducible. Then  either the action of $G^0$ on $V$
can be identified with the action of the maximal 
torus of~$\SU n$ on~$\C ^n$, or~$G^0$ is one of the groups 
$\U 2$ or~$\U 1 \cdot \SP 2$ and~$V$ is the double 
of the vector representation (on~$\C ^2$ or~$\mathbf H ^2$, 
respectively).
\end{prop}

\subsection{Basic lemma}  
Since the representation has trivial copolarity, it has trivial 
principal isotropy group. Thus 
$G/G_0$ is generated by nice involutions in terms of 
Subsection~\ref{newsubsec}.
Since $G^0$ is normal in $G$, any element of $G$ sends a 
$G^0$-irreducible subspace
to another $G^0$-irreducible subspace.  
Thus we obtain an action of the finite group $\Gamma := G/G^0$
on the set of $G^0$-isotypical components and on the set 
of $G^0$-irreducible subspaces.

Recall that, by assumption, our group $G$ has positive dimension.    
Since its representation on $V$ has trivial copolarity, it is not polar.

The basic step is the following observation: 
 
\begin{lem}  \label{keyreducible}
Under our general assumptions,  let 
 $U_{\pm 1}$ be $G^0$-invariant subspaces 
 with $U_1 \cap U_{-1} = \{ 0 \}$.
Let $w\in G$ be a  nice involution that satisfies $w(U_{-1} )=U_1$.
Then the action of $G^0$ on $U_{\pm 1}$ is of cohomogeneity $1$.
\end{lem}

\begin{proof}
Denote by $F$ the subspace of fixed points of $w$ in $V$.
By Subsection~\ref{newsubsec}   the subset $G\cdot F$  
must be of codimension $1$ in $V$.
In particular, the subset $G^0\cdot F_0$ must be of codimension $1$ in 
$U_1\oplus U_{-1}$,
where $F_0$ is the space of all $(u + w(u))$ for $u\in U_1$. 
However,
$G^0\cdot F_0$ is an algebraic set that is contained in the subspace 
$\Delta$ of all
$u+ v \in U_1 \oplus U_{-1}$ with $|u|=|v|$. Hence, $G^0\cdot F_0$ 
contains an open subset
of $\Delta$. 

Thus, for some unit  $u \in U_1$ and all $v$ in an open subset of the 
unit sphere of  $U_{-1} $   there is some $u'\in U_1$ and $h\in G^0$ with 
$hu' = u$ and $hw u'= v$. In particular, we have $hwh^{-1} u =v$, thus
 $(whw) h^{-1} u = wv$.  Hence, the orbit $G^0 u$ contains an open subset 
of the unit 
sphere in $U_1$. If $\dim U_1\geq 2$, then  
$G^0$ acts transitively on the unit sphere in $U_1$.  
If $\dim (U_1)=1$, the statement is clear anyway.
\end{proof}

\subsection{Isotypical components}
Recall that the action of $G$ on $V$ is irreducible. 
Hence  the action of $\Gamma =G/G^0$ on the set
of $G^0$-isotypical components is transitive. From this 
and \lref{keyreducible} we derive:

\begin{lem} \label{isotypical}
Under the assumptions above, let $V= V_1 \oplus\cdots\oplus V_l$ be 
the decomposition
of $V$ into $G^0$-isotypical components, and assume that $l>1$. 
Then the action
of $G^0$ on each $V_i$ is of cohomogeneity $1$; in particular, 
each $V_i$ is $G^0$-irreducible.  The group  $\Gamma$ acts on the 
set $S= \{ V_1,\ldots,V_l \}$ of $G^0$-isotypical components as the 
full permutation group.
\end{lem}

\begin{proof}
If there is a trivial $G^0$-isotypical component, 
then from the transitivity of the action of $\Gamma$,
we deduce that all isotypical components are trivial. 
Since the action of $G$ is effective, we would get that
$G$ is discrete, in contradiction to our assumption.

The group $\Gamma$ is generated by nice involutions. 
 Since the action of $\Gamma$ on~$S$ is transitive, 
for any $V_i$, we find a nice  involution $w\in G$ which moves $V_i$ to 
some~$V_j \neq V_i$. 
We use \lref{keyreducible} to see that the action of $G$ on~$V_i$ 
has cohomogeneity~$1$.
  
We claim that  such $w$ leaves all other $V_k$, $k\neq i$, $j$ invariant. 
Otherwise, we could set $U_{-1} := V_i \oplus V_k$ in 
\lref{keyreducible} and obtain that~$G^0$ acts with cohomogeneity~$1$ 
on~$V_i \oplus V_k$, which is impossible.
Thus the action of the group~$\Gamma$ on the finite set~$S$ 
is generated by transpositions. Since it is 
also transitive, the image of $\Gamma$ must be the full group of 
permutations of~$S$.
\end{proof}

In the same way we are going to deduce:

\begin{lem}  \label{onecomponent}
Assume that there is only one $G^0$-isotypical component.  Then the action
of $G^0$ on each $G^0$-irreducible subspace of $V$ is of cohomogeneity $1$.
Moreover, there is a nice  involution $w\in G$ and a decomposition of 
$V=V_1 \oplus \cdots \oplus V_l$
into $G^0$-irreducible subspaces $V_i$ such that $w(V_1)=V_2$ and 
$w(V_i)=V_i$, for all $i \geq 3$. 
\end{lem}

\begin{proof}
Take a $G^0$-irreducible subspace $V_1$.  Since $V_1$ is not $G$-invariant, some element of $\Gamma$ moves $V_1$ to another $G^0$-irreducible subspace.
We find a nice  involution~$w \in G$ that 
does not leave~$V_1$ invariant.  
We may apply \lref{keyreducible} and deduce that the action
of $G^0$ on $V_1$ and hence on any $G^0$-irreducible subspace has 
cohomogeneity~$1$.
Set $V_2=w(V_1)$ and choose pairwise  orthogonal 
$G^0$-invariant subspaces $V_3,\ldots,V_l$ that are orthogonal to 
$V_1$, $V_2$ and satisfy $V=V_1\oplus\cdots\oplus V_l$.  
By the same argument as in the proof of Lemma~\ref{isotypical}, 
we deduce that $w$ must leave the $V_i$, $i\geq 3$, invariant.
\end{proof}

\subsection{Hopf action and its brothers} \label{Hopfbrother}
In this subsection we are going to analyze the case of one isotypical component.
We start with the following simple observation:

\begin{lem}   \label{trivialcopolarity}  
Let a connected group $K$ act effectively with cohomogeneity~$1$ 
on a vector space $U$.
If the doubling representation of $K$ on $U\oplus U$ has trivial 
copolarity
 then  $K$ is one of the classical groups $\U1$, $\SU2 =\SP1$,
$\U2$, $\SP2$, $\U1\cdot \SP2$
 with its vector representation on $U=\C ^1$, 
$\C^2 =\mathbf H$, $\C^2$, $\Q^2$, $\mathbf H^2$, respectively.
 \end{lem}

\begin{proof}
The representations of cohomogeneity $1$ are listed 
in~Subsection \ref{transac}.
Going through the list, one observes  
that if $K$ acts on $U\oplus U$ with trivial principal isotropy 
group then $K$ is either one of the groups above, 
or $K=\SO3$, or $K=\SU3$.

However, the action of $\SO3$ (resp.~$\SU 3$) on $\R^3\oplus\R^3$
(resp.~$\C ^3 \oplus \C ^3$) is 
orbit equivalent to the action  of $\OG3$ (resp.~$\U 3$). 
Thus it has non-trivial copolarity.
\end{proof}

\begin{prop}
Under the general assumptions of this section, assume that there is only one 
$G^0$-isotypical component in $V$.   Then $G^0$ is  
one of the groups  $\U1$, $\U2$, $\U1\cdot\SP2$, $V=V_1\oplus V_2$,
and $V_1$, $V_2$ are the vector representations
of $G^0$, i.e., $\mathbf C$, $\mathbf C^2$, $\mathbf H^2$, respectively.
\end{prop}

\begin{proof}
Consider an involution $w$ in $G$  
and a decomposition  $V=V_1 \oplus\cdots\oplus V_l$
into $G^0$-irreducible subspaces, as in \lref{onecomponent}.  
Since $G$ acts effectively,
the assumption that there is only one isotypical component implies that 
$G^0$ acts effectively on $V_1$. We identify $G^0$ with its image in 
$\SO{V_1}$ and recall that it acts on $V_1$ with cohomogeneity~$1$. 
Since there is only one isotypical component, we fix ($G^0$-equivariant) 
identifications of $V_i$ with $V_1$ for all $i$.

We write elements in $V$ as $(v_1,\ldots,v_l)$, 
for $v_i\in V_i =V_1$.
Then the action of $G^0$ on $V$ is given in these coordinates by 
$g\cdot(v_1,\ldots,v_l)=(g \cdot v_1,\ldots,g\cdot v_l)$. 
Moreover, by assumption,
there are isometries ~$p_1,\ldots,p_l:V_1\to V_1$    
such that the action of $w$ is given by
$$w(v_1,\ldots,v_l) = (p_2(v_2),p_1(v_1),p_3(v_3),\ldots,p_l(v_l)).$$

Since $w$ is an involution, we have $p_2=p_1 ^{-1}$ and $p_i ^2 =1$, for $i\geq 3$. We set $p=p_1$. For an element $g\in G^0$, 
the conjugation $g^w=w^{-1} g w$ acts as
$$g^w (v_1,\ldots,v_l)=(g^p (v_1),g^{p^{-1}} (v_2), g^{p_3} (v_3),...,g^{p_l} 
(v_l)).$$
 
By assumption, $w$ normalizes $G^0$.  Thus we must have 
$g^p = g^{p^{-1}}= g^{p_i}$, for all $i\geq 3$.   
And this element $g^p$ is contained  in $G^0$.
Denoting by $N$ and $C$ the normalizer and the centralizer of 
$G^0$ in $\OG{V_1}$,
we deduce: $p$, $p_i\in N$ and $p^2$, $pp_i^{-1}\in C$ for $i\geq3$. 
  
Let $F$ be the subspace of fixed points of $w$ in $V$. 
It consists of all elements of the form
 \[ (v_1,p (v_1),f_3,\ldots,f_l), \] where $v_1\in V_1$ is arbitrary 
and $f_i \in V_1$ is fixed by $p_i$.  Since~$w$ is a 
nice involution, the set 
$\mathcal S =G^0 \cdot F$ is of codimension $1$ in $V$.
 
First, assume  $l\geq 3$. 
Then there is some $c\in C$ with $p =c\cdot p_3$. 
Thus any element $(u_1,\ldots,u_l) \in \mathcal S$ has the 
first three coordinates given
by $(u_1,u_2,u_3) =(g(v),g(c\cdot p_3 (v)),g(f_3))$, 
for some $g\in G^0$,  some $v\in V_1$ and some $f_3 \in V_1$  
fixed by the involution $p_3$.  
 
Hence we have $|u_1|=|u_2|$ and $(u_2-c(u_1))\perp c(u_3)$.
The assumption $\dim \mathcal S=\dim V-1$ implies $u_2 =c(u_1)$, and, therefore,
$\dim (V_1)=1$.  Then $G$ is discrete in contradiction to our assumption.
We deduce $l\leq 2$, so $l=2$.
  
Using \lref{trivialcopolarity} and our assumption 
that $G$ acts with trivial copolarity,  
we deduce that if the lemma does not hold, 
then $G^0=\SP m$ for $m=1$ or $m=2$ and $V_1 = \mathbf H^m$.
Thus we only need to exclude these two cases.

Under the assumption $l=2$, we   have  
\[ \mathcal S=G^0\cdot F= \{(v,p^g (v))\; |\; g\in G^0,\, v \in V_1 \}. \]
We have 
$$\langle p^g(v), v \rangle  = \langle(p^g)^2 (v), p^g (v)\rangle = 
\langle (p^2) ^g (v) , p^g (v)\rangle =
\langle p^g (v), p^2 (v)\rangle,$$ 
since $p^2 \in C$.  Thus  all elements $(u_1,u_2)$ in $\mathcal S$ satisfy
$|u_1|=|u_2|$ and $u_2 \perp  (u_1 - p^2 (u_1))$. 
Since $\mathcal S$ has codimension $1$, we deduce  that $p^2$ must be the 
identity.

Let now $c$ be an arbitrary element in $C$.  
Then $cp \in N$ and $(cp)^2=c(pcp^{-1})p^2 \in C$.
The same calculation as above reveals,  
$(cp)^g (v) \perp (v - (cp)^2 (v))$, hence
$p^g (v) \perp c^{-1}  (v- (cp) ^2 (v) )$.  
Again the assumption that $\mathcal S$ has codimension $1$ implies 
that $(cp)^2$ is the identity. Thus we have shown that 
for all $c\in C$ the equality $(cp)^2 =1$ must hold.

If now $G^0= \SP m$ then  its normalizer is $N = \SP m \cdot \SP1$ and 
its centralizer is $C = \SP1$. 
If the involution $p\in N$ is given by 
$\bar A \cdot \bar p$, for $\bar A \in \SP m$, $\bar p\in \SP1$, then we 
have $\bar A ^2 = \pm 1$. 
From above we deduce that for all $c \in \SP1$
we must have $(c \cdot \bar p ) ^2 =\pm 1$. But this is impossible.
 \end{proof}

\begin{rem}
Unfortunately, by this type of arguments it is impossible to exclude the 
remaining cases $G^0= \U2, G^0 =\U1 \cdot \SP 2$.
In fact, one can find an extension of the action of such $G^0$ on 
$V_1\oplus V_1$
by an involution $w$ such that
$w$ acts as a reflection on $(V_1\oplus V_1)/G^0$.
\end{rem}

We finish this subsection by noting that the Hopf action of $\U 1$ 
on $\C ^2$ is equivalent to 
the action of the unit torus of $\SU 2$ on $\C ^2$. Thus we have 
finished the proof of \pref{mainproposition},
under the additional assumption that the action of $G^0$ has only 
one isotypical component.

\subsection{Generalized toric actions I}  \label{gen1}
Now we proceed with the examination of the case of several isotypical 
components. In this and in the next subsection
we will work under the general assumptions of the present section, 
and assume, in addition, that the action
of $G^0$ on $V$ has several isotypical components. We will denote by 
$V=V_1 \oplus\cdots\oplus V_l$ the unique decomposition into 
isotypical components. Due to \lref{isotypical}, the group  $G^0$ 
acts on all~$V_i$ irreducibly and with cohomogeneity~$1$.
Moreover, the group $\Gamma =G/G^0$ acts on $S= \{ V_1,\ldots,V_l\}$ 
as the full group of permutations. 

\begin{prop} \label{torus}
Assume, in addition, that $G^0$ is a
$k$-dimensional
torus $T$. Then the action of $G^0$ on $V$ can be identified with
the action of the maximal torus of the special unitary group $\SU {k+1}$
on $\C^{k+1}$.
\end{prop}

\begin{proof}
Since all irreducible non-trivial representations of $T$ have dimension~$2$,
all $V_i$ are two-dimensional.   Thus each $V_i$ is given by a character 
$\phi _i :T \to S^1$ which is defined up to  conjugation,
i.e., it is given by a (integral) linear functional $d\phi_i$ 
on the Lie algebra $\Lt$ of $T$, which is defined up to sign.
Since each irreducible summand of $V$ is an isotypical component,
each $d\phi_i$ occurs exactly once, and $\Gamma = G/G^0$ acts 
as the full permutation group on the set of all $d\phi_i$.

We fix a $\Gamma$-invariant flat metric on $T$. 
Then the $d\phi_i$ have all the same length and we can identify 
each $d\phi_i$ with the line $t_i$ which is orthogonal to the 
kernel of $d\phi_i$ in $\Lt$.
The image $Q$ of $\Gamma$ in the orthogonal group $\OG k= 
\OG{\mathfrak t}$
permutes the $l$ elements $t_i$ of the projective space 
$\R P^{k-1} = P (\mathfrak t)$.  

Since the action of $T$ is effective, the lines $t_i$ span $\mathfrak t$, 
thus  $l\geq k$.
If $l=k$, then $T$ is a maximal torus of $\OG V$ and the action of $T$ on 
$V$ is polar 
(hence of non-trivial copolarity). Thus $l\geq k+1$. 
As we have seen, $Q$ permutes the finite 
set  $S$ of all $t_i$ and acts as the full permutation group of this set~$S$. 
Thus the distance between each pair of different $t_i$, $t_j$ in 
the projective space does not depend on $i$ and $j$. Therefore the 
lines $t_i$ are \emph{equiangular}
in the Euclidean space $\mathfrak t$. In general,
the sets of equiangular lines can be quite large and difficult to 
understand (cf.~\cite{equiang}). However, 
in the presence of the
large group $Q$, they can be described quite easily:
 
\begin{lem}
Let $\mathfrak t$ be a $k$-dimensional Euclidean vector space. 
Let $S$ be a set of lines ($1$-dimensional subspaces)
in $\mathfrak t$ that consists of $l>k$ elements and generates $\mathfrak t$.
Let $Q$ be a finite subgroup of $\OG{\mathfrak t}$ that leaves~$S$ 
invariant and acts on~$S$ as the full group of permutations of~$S$.  
Then $l=k+1$ and the lines of~$S$ are generated by the vertices of a 
regular simplex centered at the origin  in~$\mathfrak t$. 
\end{lem}
  
\begin{proof}
We proceed by induction on $k$. For $k=2$ the claim is easily 
verified (for $k=1$, such sets of lines do not exist).
Assume the statement is already shown for the dimension $k-1$,
for some $k\geq3$. 

Let $\alpha \leq  \pi /2$ be the angle between any pair of lines from $S$, 
which by assumption does not depend on the
pair. If $\alpha =  \pi /2$, all lines are pairwise orthogonal, thus 
the number of lines is bounded from above by $k$,
in contradiction to our assumption.  If $3$ lines $t_1$, $t_2$,
$t_3 \in S$ 
lie in a plane, we must have $\alpha = \pi /3$.
Choose another line $t_4$. We can choose unit vectors $X_i$ on $t_i$ such 
that  the spherical distances (i.e., the Euclidean angles)  satisfy 
$d(X_1,X_2)=d(X_2,X_3)=d(X_2,X_4)= \pi /3$. Since $X_4$ does not 
coincide with $\pm X_1$ or $\pm X_3$, its distances
to $X_1$ and $X_3$ are less than $2\pi /3$. Thus they must be equal 
to~$\pi/3$. But such quadruple of points does not exist
in the unit sphere $S^2$.

Thus $\alpha <\pi /2$ and any three lines in $S$ 
generate a $3$-dimensional space.  
Consider the subgroup $Q_1$ of $Q$ that
 leaves $t_1$ invariant.  Then it acts on $\mathfrak t_1$, the orthogonal 
complement of $t_1$ and it preserves
 the set $S_1$ of lines in $\mathfrak t_1$ that are  projections of lines 
in $S$ different from $t_1$.
 Since no three lines lie in a plane, projections of different lines are  
different, thus $S_1$ has $k-1$ elements.
 By assumption, $Q_1$ acts as the full permutation group
of  $S_1$. By our inductive 
assumption, $l-1=(k-1)+1$. Hence $l=k+1$.
 Moreover, by induction, the lines in $S_1$ are given by the vertices 
$Y_2,\ldots,Y_l$ of a regular simplex in $\mathfrak t_1$.
 
Fix a unit vector $X_1$ on $t_1$. 
Choose the unit vector $\bar X_i$ on $t_i, i=2,\ldots,k+1$, 
with $d(X_1,\bar X_i)=\alpha$. Then $\bar X_i$ lies on
the spherical geodesic from $X_1$ to $Y_i$.  
Moreover, by our assumption, for all $i\neq j$, we have either 
$d(\bar X_i,\bar X_j) =\alpha$ or $d(\bar X_i,\bar X_j)= \pi- \alpha$. 
However, the angle between $X_1Y_i$ and $X_1Y_j$ is larger than $\pi /2$.
Thus the triangle $X_1 \bar X_i \bar X_j$ cannot be equilateral. 
Hence $d(\bar X_i, \bar X_j)= \pi -\alpha$ for all 
$i\neq j$. Therefore if we set $X_i= -\bar X_i$ for $i\geq2$,
the $(k+1)$ unit 
vectors $X_i$ have the same pairwise distances, given by
$\pi -\alpha$. Thus $X_i$ are the vertices of a regular simplex.  
\end{proof}
  
Now we can finish the proof of \pref{torus}.
We deduce form the last lemma that $l=k+1$ and that, 
up to an eventual change of $\phi _i$ by its conjugate 
(that does not effect the representation of the torus $T$), 
the homomorphisms $\phi _i :T\to S^1$  have differentials $d\phi_i$ that
are the vertices of a regular simplex. Hence these differentials  
satisfy $d\phi_1+d\phi_2+...+d\phi_{k+1} =0$.  
But this is exactly the defining 
equation of the maximal torus of $\SU {k+1}$.
\end{proof}

\subsection{Generalized toric actions II} \label{gen2}
In this subsection we are going to finish the analysis of the 
situation where the action of $G^0$ on $V$ has several isotypical 
components $V=V_1\oplus\cdots\oplus V_l$. Here we assume that $G^0$ 
is not commutative. Since all isotypical components are permuted  
by $\Gamma =G/G^0$, the images~$K_i$ of~$G^0$ in the isometry group of 
the isotypical components $V_i$ are congruent inside $\OG V$.

Denote by $K'$ the image of $G^0$ in $\OG{V_1}$. Since $G^0$ is contained 
in a product of groups isomorphic to $K'$, the group $K'$ is not commutative. 
Going through the list of groups $K'$ acting with cohomogeneity one on a 
vector space, we get three cases that we will analyze separately: 
$K'$ is covered by a simple group $K$; $K'$ is covered by a group $K$
with two different factors; $K'$ is covered 
by $K=\SP 1 \times \SP 1$.

{\bf Case 1.}
Let us assume that $K$ is a  simple group.
Then $G^0$ is a finite quotient of a product $K^m$.
Moreover, for any component $V_i$, exactly one of the $m$ factors 
is mapped not to the identity in 
$\SO{V_i}$.  If $K\neq\Spin8$ then there is exactly one  
irreducible representation of $K$ with cohomogeneity $1$
on a vector space of the fixed dimension 
$\dim V_1=\dim V_2=\cdots=\dim V_l$.
Thus, in this case, $m=l$ and 
for each $j=1,\ldots,m$ 
there is exactly one $V_i$ such that the $j$-th factor $K$ is 
mapped non-trivially into $\SO{V_i}$. 
We deduce that the action of $G^0$ on $V$ is an $l$-fold  
product of the actions of $K$ on $V_1$.  
However, this action is polar.  This contradicts our assumption.

If $K=\Spin8$ then $K$ has $3$ representations 
$\rho_1$, $\rho_2$, $\rho_3$ of cohomogeneity~$1$ on 
$\mathbf R^8$ that are permuted by the triality automorphism.  
Then the only part of the representation of $G^0$ on which
a fixed factor $K$ may act non-trivially is a sum of different  
$\rho_j$, $j=1$, $2$, $3$, i.e., a vector space of dimension at 
most $24$. However, $\dim\Spin8=28$. 
Thus the action of $K$ on this subspace has non-trivial principal isotropy groups.  Then the action of $G^0$ on $V$
has non-trivial principal isotropy groups, as well. 
Thus it has non-trivial copolarity in contradiction to our assumption.

{\bf Case 2.} 
Now, we assume that the image $K'$ is covered by a product $K$ of two 
different factors (i.e., $\U1\times\SU n$, 
$\U1\times\SP n$, $\SP1\times\SP n$, 
for $n\geq 2$). In all cases,
the action of one factor $\hat K$ of $K$ 
is still of cohomogeneity $1$.
Consider the connected normal subgroup $N$ of $G^0$ 
whose Lie algebra is the sum of all factors 
isomorphic to $\hat {\mathfrak k}$, the corresponding Lie algebra. 
Since the factors of $K$ are different, $N$ is normalized by $G$. 

Consider the 
decomposition $V=W_1\oplus\cdots\oplus W_r$ into 
$N$-isotypical components.  Since $G^0$ acts on $N$ via inner 
automorphisms, any element of $G^0$ preserves
each summand $W_i$. Any element $g\in G$ permutes the $N$-isotypical 
components.  If $r>1$ then, arguing
as in \lref{isotypical},
we deduce that the action of $G^0$ on each $W_i$ has cohomogeneity $1$, 
hence each $W_i$ coincides with some $V_j$.
As in the previous case, the action of $N$ on $V$ is a direct product 
of cohomogeneity $1$ actions of $N$ on $V_i$.
We deduce that the action of $G^0$ is orbit equivalent to the action 
of $N$ on $V$ and that both actions are polar. Contradiction.

Therefore we may assume that $r=1$. Then $N$ has only one factor, i.e., 
$N$ is locally isomorphic to~$\hat K$, which is either 
$\SU n$ or $\SP  n$. The representation of $N$ on $V$ is given by  
the $l$-fold sum of the vector representation $V_1$ 
of $\hat K$.  Any 
nice involution~$w\in G$  normalizes~$N$.  
Arguing as in Subsection \ref{Hopfbrother}, we get a presentation of $w$ 
as $w(v_1,\ldots,v_l)=(p^{-1} (v_2),p(v_1),p_3 (v_3),\ldots,p_l(v_l))$,
for some involutions $p_i \in\OG{V_1}$ and some isometry $p\in\OG{V_1}$ 
such that they all normalize
$\hat K$, and such that $p^2$, $ p p_i^{-1} $ lie
in the centralizer of $\hat K$
inside $\OG{V_1}$, for $i\geq3$.
 
Assume $l\geq 3$.
The set $F$ of fixed points of $w$ is given by all 
$(v, p(v),f_3,\ldots,f_l)$ with $v\in V_1$ and the $f_i$ fixed by
$p_i$.  The extension of the action of $N$ to the action of $G^0$ 
is given by some (complex, respectively, quaternionic) scalar 
multiplications in each 
component.
Note that $p$ and $p_i$
induce the same  action  on the projective spaces $\C P^{n-1}$ and $\mathbf HP ^{n-1}$, respectively.
Thus the set  $\mathcal S=G^0\cdot F$ is contained  
in the set of all points with 
the first three coordinates $(u_1,u_2,u_3)$ given by  
$(g(v)\lambda_1, g(p_3(v))\lambda_2, g (f_3 )\lambda_3)$  for some  
scalars~$\lambda_i$, some element $g\in \hat K$, 
some $v\in V_1$ and some~$f_3$ 
in~$V_1$ fixed by~$p_3$.
  
We have $|u_1|=|u_2|$. 
Moreover, the ``lines'' $[u_1]$, $[u_2]$, $[u_3]$  
(i.e.~the elements of the corresponding projective space)
are given by $g([v])$, $g (p_3 [v])$, $g ([f_3])$, 
where~$[f_3]$ is fixed by the involution $p_3$ in the projective space.
Thus 
$d([u_1],[u_3]) =d([u_2],[u_3])$ in the projective space $P(V_1)$ (over $\C$ and $\mathbf H$, respectively). 
Hence the set $\mathcal S$ has codimension at least~$2$ in $V$, and 
$w$ cannot act as 
a reflection. This provides a contradiction in the case $l>2$. 
 
We deduce $l=2$.  Since the  
action of $G^0$ has trivial copolarity,  the action of $G^0$ and therefore 
of $N$  on $V=V_1 \oplus V_2$
has trivial principal isotropy groups. From the classification of cohomogeneity $1$ actions, we deduce that
$\hat K= \SU n$, with $n=2$ or $3$, or $\hat K= \SP 2$.
  
 If $\dim (G^0 /N ) \geq 3$ then $\dim (V) -\dim (G^0 ) \leq 3$ in all cases. Thus, due to \cite{str},  the action of $G^0$
cannot have trivial copolarity. Hence $\dim (G^0 / N) \leq 2$.  If $\dim (G^0 /N) =1$, then, using the fact that the
representations  $V_1$ and $V_2$ of $G^0$ are congruent 
inside $G$, we see that both representations are equivalent.
Hence there is only one $G^0$-isotypical component, in contradiction to our assumption. Therefore $G^0 /N$ is the 
two-dimensional torus $T^2$.  If $\hat K = \SU n$ 
we get $\dim (V) -\dim (G^0 ) \leq 3$ again, in contradiction to 
the trivial copolarity assumption.   
Thus $G^0$ must be finitely covered by $\SP 2\times\U1^2$.
Now $V_1=V_2=\Q^2$ and we use the canonical real basis $\{1,i,j,k\}$ 
of $\Q$. Up to orbit-equivalence, we may assume that 
each $\U1$-factor acts on the corresponding summand $\Q^2$ by right 
multiplication by matrices of the form
$e^{i\theta}\,\mathrm{id}$. Now $\sigma(v_1,v_2)=(-jv_1j,-jv_2j)$,
for $(v_1,v_2)\in V_1\oplus V_2$, is an involution that preserves
each $G^0$-orbit. It follows that the action of $G^0$
on $V$ has non-trivial copolarity in contradiction to our assumption.

{\bf Case 3.}
Now assume that the image is covered by a product $K$ 
of equal factors.  Then, $K$ is $\SP1\times\SP1$, 
$\dim V_i=4$ for all $i$, and $G^0$ is mapped onto $\SO4$ in this case. 
Thus $G^0$ is covered by a product $\SP1^m $ for some $m$.

Denote as above by $\Gamma$ the group 
$G/G^0$. This group permutes the $m$ factors of 
(the universal covering of)  $G^0$, hence is mapped onto a subgroup
$\Gamma '$ of the symmetric group $S_m$.  Note that on each 
$V_i$ exactly $2$ factors of $G^0$ act non-trivially.
Thus the isotypical components  can be indexed by a subset 
$A$ of the set of  unordered pairs of different  factors.
We will denote such unordered pairs by $i\otimes j$,
$1\leq i,j \leq m$. The action 
of $\Gamma$ on the set of $G^0$-components factors through 
$\Gamma'$.  Moreover, the action of $\Gamma'$  on $A$ is induced 
by the canonical action of $S_m$ on the set of unordered pairs.  
Due to \lref{isotypical},
the image of $\Gamma'$ is the whole permutation group
of~$A$.

The cardinality of $A$ is $\ell$. Since $\Gamma'$ acts as the 
full permutation group of $A$, we have~$\ell\leq m$. 
For any $i=1,\ldots,m$, we set $A_i = \{ j \;|\; i \otimes j \in A \}$ 
and denote by $a(i)$ the cardinality of $A_i$.
If $a(i)\geq 2$, for all $i$, then $A$ has at least $m$ elements, 
so $\ell=m$. It follows that $\Gamma'=S_m$ and 
hence $A$ has $m(m-1)/2=m$ elements. Hence $m=3$. Then $\dim (V)=12$ and 
$\dim (G^0 )=9$. Thus the action of $G^0$ on $V$ cannot have trivial copolarity by \cite{str}. 

Hence there exists some $i$ with $a(i)=1$.  If  $a(i)=1$, for all $i$,
then the action of $G^0$ is the direct product of the cohomogeneity $1$ representations of $G^0$ on $V_i$.
Thus the representation is polar in contradiction to our assumption. 

Thus, there are some $i$ with $a(i)=1$ and some $j$ with $a(j)>1$.
Note that the number $a(i)$ 
is constant on orbits of $\Gamma '$. 
 Since $\Gamma'$ acts transitively on $A$, 
for all pairs $i\otimes j \in A$, the unordered pair
$a(i) \otimes a(j)$ does not depend on $i\otimes j$. 
Hence  there are exactly two orbits $B$ and~$C$ of~$\Gamma'$
in $\{ 1,\ldots,m \}$, one of them, say~$B$, consisting of all~$i$ 
with $a(i)=1$, and the other one $C$ consisting of all
elements~$j$ with~$a(j) =a>1$.
 
Assume that $C$ has at least two elements $j_1$, $j_2$. Choose 
distinct elements
$i_{\pm} \in B$ such that $j_1 \otimes i_{\pm}  \in A$.
Then there is no element in $\Gamma '$ that leaves 
$j_1 \otimes i _+$ 
invariant and moves $j_1 \otimes i _-$
to some pair $j_2 \otimes i$. Contradiction to the fact that 
$\Gamma '$ acts as the full permutation group on $A$. 
  
Thus, $C$ has only one element that we may assume 
to be $\{ m \}$. Then $B$ has $m-1$ elements, and the action of the 
first  $(m-1)$ factors on $V$ is a (polar) product action.
One sees that the action of $G^0$ is in this case polar as well 
(and has non-trivial principal isotropy group).  
 
This finishes the proof in the case the image of $G^0$ in $V_1$ 
is covered by a product of equal factors.  

Taking all pieces together, we have completed the proof of \pref{mainproposition}.

\section{An example} \label{secexample} 
Before we go on to prove \tref{special}, we need to collect a few 
observations
about the geometry of the action of a maximal torus $T$ of $\SU{k+1}$
on $V=\C^{k+1}$.

 \begin{lem} \label{noboundarytorus}
The quotient $V/T$ has no boundary.
 \end{lem}

\begin{proof}
No circle inside $T$ fixes a subset $F$ of complex codimension $1$. 
Since $T$ is commutative we derive the result from the dimension formula 
(Lemma~\ref{boundreduced}).
\end{proof}

In particular,  the action of any finite extension of $T$ has 
trivial abstract copolarity, due to \pref{boundaryexist}.

Note that the normalizer $N=N(T)$ of $T$ in $\mathbf O (2k+2)$  
normalizes the centralizer $\bar T$
of $T$ in $\mathbf O (2k+2)$. 
Assume $k\geq2$. Since $\bar T$ is a maximal 
torus of $\mathbf  U (k+1)$
which is also a maximal torus of $\mathbf O (2k+2)$, we see that 
$\bar T$ is the identity component of the normalizer $N$. Moreover, $N$ 
is generated by $\bar T$, the complex conjugation $c$ and the 
symmetric group $S_{k+1}$ of permutations of complex coordinates.
On the other hand, the case $k=1$ is discussed in~\cite[p.14]{str}.

\begin{lem} 
Assume $k\geq1$. Then the
action of the normalizer $N(T)=N$ of~$T$ in~$\mathbf O (2k+2)$ on 
the quotient $V/T$ induces an isomorphism
of $N/T$ with the isometry group $Iso(V/T)$  of $V/T$.
\end{lem}

\begin{proof}
For $k=1$, the result is contained in \cite[p.14]{str}.
Thus we will assume $k\geq 2$.

If an element of $N\setminus T$ acts trivially on $V/T$ we obtain an 
action by a group $T'$  larger than $T$ that is orbit equivalent to  the 
action of $T$. This would imply that the action of $T'$ had non-trivial 
principal isotropy group. Hence the action of $T'$ (and thus of $T$) had 
non-trivial copolarity. Contradiction. Thus the map from $N/T$ to 
$Iso (V/T)$ is injective.

To prove the surjectivity, take an isometry $I$ of $V/T$ to itself. 
Restrict it to the unit sphere
$X=S(V) /T$.   There are $k+1$ irreducible subspaces of $V$, 
each of them defines a unique point
$p_i$ in $X$, for $i=1,...,k+1$.  
Due to  \lref{folk}, these points are permuted by $I$.   
Composing $I$ with an element of $N$ that permutes the complex coordinates 
backwards, we may assume that $I$ fixes the points $p_i$.

Let $N'$ denote the subgroup of the isometry group of $V/T$ that fixes 
all the points $p_i$. The intersection of $N'$ with the image of $N$ is 
the group $\OG2$ generated by the image of $\bar T$ and the 
complex conjugation.  
We have to prove that $N' =\OG2$. In order to do so, consider the map
$F=(d_{p_1},d_{p_2},\ldots,d_{p_{k+1}}) :X \to \R ^{k+1}$ whose 
coordinates are distance functions to the points
$p_i$.  By construction, the function $F$ is $N'$-invariant.  
We claim that the fibers of $F$ are the orbits
of $\bar T$. Namely, due to the $N'$-invariance, the orbits of 
$\bar T$ are contained in the fibers of $F$.
On the other hand, $Y=X/\bar T = S(V)/\bar T$ is the rectangular  
spherical simplex with vertices given by the points $p_i$. And the 
function $F$  that descends to $Y$ separates the points of $Y$.

Thus the regular fibers of $F$ are circles. The group $N'$ acts 
effectively on $X$ having circles as regular orbits. 
This implies that $N'$ can be only $\U 1$ or $\OG 2$. 
Since its intersection with
$N$ is already $\mathbf O (2)$,  the whole group $N'$ must coincide with 
this intersection.
\end{proof}

 As a consequence we deduce:

\begin{cor} \label{conjugate}
 Let $T_1$, $T_2$ be finite extensions of $T$ in $\OG{2k+2}$. 
If $V/T_1$ and $V/T_2$ are isometric,
then $T_1$ and $T_2$ are conjugate inside  $N(T)$.
\end{cor}

\begin{proof}
 Consider an isometry $J:V/T_1 \to V/T_2$. Since $V/T$ does not have 
boundary, $(V/T)_{orb}$ is the universal orbifold covering of 
$(V/T_i) _{orb}$ (Lemma~\ref{simplycon}). 
Then the isometry $J$ is induced by an 
isometry $I:V/T \to V/T$, and we can use the preceding lemma 
to lift $I$ to an element $n$ of $N(T)$.
Now $nT_1n^{-1}/T$ and $T_2/T$ have the same orbits in $V/T$
implying that $nT_1n^{-1}$ and $T_2$ have the same orbits
in $V$. Since they both act with trivial copolarity it follows
that they are equal.
\end{proof}

\begin{lem}  \label{finallemma}
Assume $k\geq 2$.
Let $T^+$ be a finite extension of $T$ such that $T^+/T$ acts 
on $V/T$ as a reflection group and such that
$T^+$ acts on $V$ irreducibly. Then there are two 
codimension~$1$ strata in~$V/T^+$ that meet at an
angle not equal to~$\pi /2$.
\end{lem}

\begin{proof}
We have seen in \lref{isotypical} that  $T^+/T$ acts as the 
full permutation group on the components of $V$.
Thus the finite reflection group $T^+ /T$, which is the 
orbifold fundamental group of $(V/T^+)_{orb}$
surjects onto the non-Abelian symmetric group $S_{k+1}$.   
However, if all strata of codimension $1$
would only meet at right angles, the orbifold fundamental group of 
$(V/T^+)_{orb}$ would be Abelian, cf. the end of
Subsection~\ref{fundgroups}. (More directly, it is not difficult to 
observe that two reflections $w_1$, $w_2$ corresponding to non-commuting 
transpositions in $S_{k+1}$ define strata of codimension $1$ that meet at 
an angle equal to $\pi /3$.)
\end{proof}

From this we deduce our final piece:
\begin{cor} \label{finalpoint}
Assume that $k\geq 1$.
Assume that a representation $\rho :H \to \mathbf O (W)$  is quotient 
equivalent to a finite extension $T_1$
of $T$ in   $\mathbf  O (2k+2)$. If the action of $H^0$ on $W$ is irreducible, then the principal isotropy group of $H$ is non-trivial.
\end{cor}

\begin{proof}
The representation $\rho$ is irreducible. For $k=1$, it has
cohomogeneity~$3$ and thus it is listed in group~III, Table~II 
in~\cite{str}. Those have non-trivial principal isotropy 
groups.

Assume $k\geq2$. 
Due to \pref{liftnoboundary} and \lref{noboundarytorus}, 
we find a subgroup $T^+$ of $T_1$ of finite index in $T_1$ 
such that the representation of $H^0$ on $W$ and of $T^+$ on $V$ 
are quotient equivalent. Since
$H^0$ is connected, we see that $T^+/T$ acts on $V/T$ 
as a reflection group (\lref{simplycon} and \lref{keyreflection}). 
Since $H^0$ acts irreducibly, so does $T^+$ (\lref{folk}). 
Applying the previous lemma
and \lref{isotropy}, we deduce that $H^0$ (and therefore $H$) 
acts with non-trivial principal isotropy group.
\end{proof}

\begin{rem}
We would like to mention that \lref{finallemma} does not hold for $k=1$,
see~\cite[Table~II]{str}.
\end{rem}

\section{Conclusion} \label{secconclusion} 

Now we are going to prove~\tref{special}.

Consider  the action of $H'$ on  $W'$ 
and take its reduction to a minimal generalized section~$V$.
Let $\tau :G \to \mathbf O(V)$ be this reduction. 
By construction it has trivial copolarity
and is in the same quotient equivalence class as 
$\rho :H\to \mathbf O (W)$.  Since $V$
is also a generalized section of $(H') ^0$, we see 
that the action of $G^0$ on $V$ is reducible. 

Due to  \pref{onereflection}, we find  subgroups of finite index  
$H^+$ in $H$ and  $G^+$ in $G$
such that the corresponding restricted representations are still 
quotient equivalent and such that
$G^+/G^0$ acts on $V/G^0$ as a reflection group.  
Since $H^+$ contains $H^0$ it acts irreducibly and so 
does~$G^+$, due to \lref{folk}.  
In particular, $G^+/G^0$ is a non-trivial group. 
Now, if $G$ is non-discrete, 
we are in the situation of \pref{mainproposition}.
We deduce that either the action of $G^0$ on $V$ is as required,
or that $G^0$ is one of the groups  $\U 2$ or $\U 1 \cdot \SP 2$ 
and $V$ is the double 
of the vector representation (on $\C ^2$ and $\mathbf H ^2$, respectively).

Assume that $G^0 = \U 2$. Then the action of $G$, hence that of~$H$ 
has cohomogeneity $4$.
We use the explicit classification of all irreducible representations 
of cohomogeneity $4$ of connected groups (\tref{mainclass}
that is proven independently in the second part of the paper).  
If $H^0$ is~$\SO 3$ or~$\U 2$
(the first two lines in the table), then $W/H^0$ has no boundary. 
Due to \pref{liftnoboundary},
we find a subgroup $H_1$ of finite index in $H$ such that $W/H_1$ 
and $V/G^0$ are isometric.
But $H_1$ acts irreducibly and $G^0$ 	
acts reducibly, in contradiction to \lref{folk}.
On the other hand, if $H^0$ has non-trivial copolarity, 
then it has copolarity $2$, according to \tref{mainclass}. 
Since the copolarities of $H$ and $H^0$ coincide, we find a reduction of $H$
to a representation of a group $K$ on $U$ with $K^0 =T^2$.  
The action of $K$ has trivial copolarity
and the representation of $K^0$ is reducible. Thus applying the 
arguments we have applied to 
$G$ (Propositions~\ref{onereflection} 
and~\ref{mainproposition}), 
we see that $K^0$ is the maximal torus of $\SU 3$ and $U=\C ^3$. 
Since $U/K^0$
has no boundary (\lref{noboundarytorus}) we 
use \pref{liftnoboundary}, to find a finite index subgroup $K_1$
of $K$ such that $U/K_1$ and $V/G^0$ are isometric.  But the action 
of $G^0$ on $V$ has infinitely many invariant subspaces 
(it is reducible with one isotypical component) 
and the action  of $K_1$ on $U$ only finitely many. 
This contradicts \lref{folk}.  
Hence the case $G^0 =\U 2$ cannot occur.

Assume now that $G^0 =\U 1 \cdot \SP 2$. 
Then the action of $G^0$ has cohomogeneity $5$, 
hence so does the action of $H$.  We use again the explicit 
classification of all irreducible  representations 
of cohomogeneity~$5$ of connected groups
(\tref{mainclass} that is proven independently in the second part 
of the paper). If $H^0$ is $\SU2$, then $W/H^0$ has no boundary
and we get a contradiction as above (Proposition~\ref{liftnoboundary}
and Lemma~\ref{folk}).
On the other hand, if the copolarity of $H^0$, hence of $H$ is $3$, then 
by the direct computations in the second part,
the reduction to the minimal generalized section has a torus $T^3$ as 
identity component.
Then arguing as in the case of cohomogeneity $4$ we obtain a contradiction. 
It remains to analyze the cases
$H^0 = \SO 3 \times \U 2$ and $H^0=\U 3 \times \SP 2$. Since  
the latter action has a reduction to the former one in terms
of a generalized section, 
we may assume that $H^0 = \SO 3 \times \U 2$.

To exclude this remaining case, we observe that $S(V)/G^0$ is a 
Riemannian orbifold. In fact it is easy to determine 
the equivalence classes of the slice representations at singular
points. There are only three possibilities, in which 
the nontrivial components of the slice representation
are respectively orbit equivalent to $(\U1,\C)\oplus(\SP1,\Q)$,
$(\U1,\C)$ and $(\SP1,\Q)$.
Since these are polar representations, the claim
follows from the main result of~\cite{LT}.

On the other hand, the 
slice representation of the identity component of the isotropy group of 
$H^0$ at a vector $v_1\otimes v_2$ is given by  
$T^2$ acting on $\C\oplus\C\oplus\C$ with weights $(1,0)$, $(1,1)$, $(1,-1)$,
hence it is non-polar. Thus $S(W)/H$ has non-orbifold points~\cite{LT}.

This proves that $G^0$ is the maximal torus of $\SU{k+1}$ and $V=\C^{k+1}$.

\begin{rem}
A more conceptual proof that the two special cases cannot 
occur can be obtained as follows. One shows that both quotient spaces
$S(V)/G^0$ are orbifolds of constant curvature, i.e., finite quotients of 
some round spheres. Hence the same is true for $S(V)/G=S(W)/H$.  One invokes a 
recent theorem of Wiesendorf \cite{Wiesen} saying that in this case 
all $H$-orbits are taut submanifolds of $W$.  Now one uses the main 
result of \cite{gt} that states that the action of $H$ on $W$ must be 
of cohomogeneity~$3$.
\end{rem}  

Assume now that there is another representation 
$\rho _1 :G_1 \to \mathbf O (V_1)$ in the same quotient 
equivalence class that has trivial copolarity.   
Due to \cref{finalpoint}, the representation of $G_1 ^0$
cannot be irreducible. Since it is reducible, we may apply 
the same arguments we have applied to $G$,
to deduce that the restriction of  $\rho _1$ to $G_1 ^0$  
is given by the action of a maximal torus
of $\SU{k'+1}$ on $\C ^{k'+1}$. Since the actions have the same 
cohomogeneity, $k=k'$ and 
the representation spaces 
of~$G_1^0$ and of~$G^0$ can be identified.  To prove that $\rho$ and $\rho _1$
are the same, we only need to invoke \cref{conjugate}.

\section{New setting} \label{secsetting}  
In
the next two sections we are going to prove \tref{verythird}  and  \cref{main}.
Thus let here and in the sequel $\rho:H\to \mathbf O (W)$ be a 
non-reduced non-polar 
irreducible representation of a connected compact Lie group. 
Let $\tau:G\to \mathbf O (V)$ be the minimal 
reduction in case of Theorem~\ref{verythird} or the reduction to a minimal
generalized section in case of Corollary~\ref{main}. 
If the action of $G^0$ is reducible, then we can apply \tref{special} 
to deduce that the copolarity and the abstract copolarity coincide and 
that the cohomogeneity is equal to $k+2$.
Thus \tref{verythird} and \cref{main} are
a consequence of the following result that will be 
proved in the next two sections.

\begin{prop} \label{insertprop} 
Assume that $\rho:H\to \mathbf O (W)$ and $\tau :G\to \mathbf O (V)$  
are quotient equivalent.
Assume that $H$ is connected, $\tau$  has trivial copolarity 
$k =\dim (G) \leq 6$ and that $\dim(H) >\dim (G)$.
Then the representation of $G^0$ is reducible.
\end{prop}

Form now on assume, to the contrary, that the action of $G^0$ is irreducible.
By assumption, the action of $G$ is effective, hence 
the Abelian summand of the Lie algebra of $G^0$ is at most one-dimensional.
Since $\dim (G) \leq 6$, the group $G^0$ is locally isomorphic to~$\SP 1$, 
or $\U 2$, or $\SP 1 \times \SP 1$, respectively.

We will distinguish two cases  depending on whether $G$ is connected or not. 
In the first case $V/G=V/G^0$ has 
boundary. We will exclude this case in the next section using a bit 
of representation theory.  If $G\neq G^0$ then
there is a nice involution 
$w\in G\setminus G^0$ 
(Proposition~\ref{onereflection}). Using the dimension 
formula for the set $F$ of its fixed points, we will
obtain a contradiction in  all but two cases that will be excluded by 
the general classification in the second part of the paper.

\section{Connected case} \label{connectedcase}
\subsection{Basics}
In addition to our assumption from the previous section, 
we assume here that $G$ is connected.  Then there is
some $G$-important point $p$ such that $G_p$ is either $\U 1$ or $\SU2$. 
In all cases ($G=G^0$ covered by $\SU 2$ or $\SU 2 \times\U1$
or $\SU 2 \times \SU 2$) any $\SU2$-subgroup contained in $G$ has a unique 
involution that is central 
in $G$. Since such a central involution cannot have fixed points 
(our representation is irreducible!), we cannot have $G_p\cong\SU2$.

We fix 
an important point $p$ whose stabilizer $G_p$ is some $\U 1$.  Let $F$ be the set of fixed points of $G_p$ and let $f$ denote its dimension.  Then  we have (\lref{boundreduced})  the dimension formula 
$f= \dim (V) - \dim (G) + \dim (N(G_p)) -2$.

\subsection{A note  on weight spaces}  
To obtain an upper bound for $f$ we make the following general  
observation that allows us to estimate $F$ using the weights of the 
representation. 

Let $\tau$ be a representation of a  
connected compact group $K$ on a complex vector space $U$. 
Let $L$ be a subgroup isomorphic to $\U1$. 
Choose a maximal torus $T$ containing $L$ . 
Let $F$ be the set of fixed points of $L$. 
Then $F$ is $T$-invariant, hence it is a sum of weight spaces. 
Note that all these weights vanish on
$L$, hence the weight spaces appearing in $F$ are associated
to weights contained in a hyperplane in the dual space 
of the Lie algebra $\mathfrak t$ of $T$. 

If $U$ is a real vector space, then one can consider its 
complexification $U^c=U\oplus iU$, and obtains that the set  
$F^c\subset U^c$ is a sum of weight spaces whose corresponding
weights are contained in a hyperplane of the dual space of $\Lt$.

If the rank of $K$ is $1$, then the only hyperplane in the 
dual of the Lie algebra of $T$ is $\{ 0 \}$,
hence $F$ (respectively $F^c$) is contained in the $0$-weight space.

Assume now that $K$ has rank $2$.   
Then all hyperplanes are one-dimensional.  
Thus all weight spaces appearing in $F$
must be linearly dependent.  
(In fact, they are all multiples of the weight given by
 $T\to T/L\cong\U 1$.)

\subsection{$G$ is covered by $\SU2$}
Here we assume that $G$ has rank $1$.   
Then all irreducible representations of $G$ are given by the 
quaternionic representations $\rho_{2n}$ on the complex space 
$\C^{2n}$ of homogeneous polynomials of degree $2n-1$ in two 
complex variables and by the real representations 
$\rho_{2n+1} $ on $\R^{2n+1}$.

From the previous subsection we see that 
$f$ is bounded from above by the real dimension of the 
$0$-weight space in the 
quaternionic case and by the complex dimension of the 
$0$-weight space in the real case.

In the quaternionic case, the $0$-weight space is trivial
(this already implies that there are no 
$G$-important points!)
and, in the real case, it is $1$-dimensional. 
Thus we get $\dim(V)=\dim(G)-\dim(N(G_p))+f+2\leq 3-1+1+2=5$.
If follows that the cohomogeneity of the action is at most $2$ and 
hence it is polar. Contradiction.

\subsection{$G$ is covered by $\U 1 \times \SU2$}
Then all irreducible representations of $G$ are complex. 
For each $n$ there is a exactly one irreducible representation
$\rho_n$ of $G$ on $\C^n$.  
The restriction of each weight to the central 
$\U 1$ of $G$ is independent of the 
weight, thus there are no linearly dependent non-equal weights.   
One knows that each weight space is complex $1$-dimensional.  
Thus we deduce $f\leq 2$. Since $\dim (N(G_p)) \geq  \dim (T) \geq 2$, 
we get from the dimension
formula $\dim (V) \leq 6$. Again the cohomogeneity 
of the action is at most $2$ 
and hence the representation must be polar.
Contradiction. 

\subsection{$G$ is covered by $\SU2 \times \SU2$}
In this case the complex irreducible representations of $G$ 
are given by $\rho_m \otimes \rho_n$ on
 $\C ^m \otimes_{\mathbf C} \C^n$.  
Using that all weights have multiplicities~$1$  
for $\rho _n$ and $\rho _m$, we see
that any number of linearly dependent weights of $\rho _m \otimes \rho _n$ 
has at most $\min \{m,n \}$ elements.
Without loss of generality, we assume $m\geq n$. 

If $m$, $n$ have different parity, then the 
representation $V$ is quaternionic and 
we get $\dim (V) =2mn$, $\dim (F) \leq 2n$.  
On the other hand we get from the dimension formula:
$2mn\leq 6 + 2n$.  Hence $n(m-1) \leq 3$, which is  
impossible as $m\geq n$.

If $m$, $n$ have the same parity, then the representation $V$ is real,  
and the dimension formula gives us $mn\leq 6 +n$.  
Hence, $n(m-1) \leq 6$. Thus either $m=n=3$ and the 
corresponding representation of $\SO3\times\SO3$
on $\R ^3 \otimes \R ^3$ is polar, or 
$m=n=2$ or $m=4$, $n=2$ which have cohomogeneity at most~$2$
and hence are also polar.   

\section{Disconnected case} \label{discconnectedcase}
\subsection{General useful observations} \label{way}
Here we are going to work with a nice involution $w\in G \setminus G^0$.
This involution normalizes $G^0$. We denote the set of the  fixed points of $w$  by $F$ and the  dimension of $F$  by $f$.
We have the dimension formula:
 $f=\dim (V) -\dim (G) + \dim(C) -1$.  Here $C$ is the centralizer of $w$ in $G$.

We will often use the following observation:

\begin{lem} \label{spn} 
Let $K\subset \mathbf{O}(U)$ be a closed subgroup such that 
$K^0$ is locally isomorphic to
$\SP n$. Assume that $K^0$ acts irreducibly
on $U$ and there is an involution $w\in K\setminus K^0$.
Then either $K^0=\SP n/\mathbf Z_2$ and $w=-w'$, for some involution 
$w'\in K^0$, or $K^0=\SP n$ and, for some complex structure on $U$,
the involution $w$ and the group $K^0$ are contained in the unitary
group of $U$. 
\end{lem}

\begin{proof}
The involution $w$ acts by conjugation as an automorphism on $K^0$. 
Since $K^0$ does not have outer automorphisms, we find some $j \in K^0$ 
such that conjugation with $j$ induces the same automorphism as 
conjugation with $w$. Then the element $q= j ^{-1} w$ commutes with $K^0$.

Since $w$ is an involution, we have $(jq)^2= j^2 q^2 =1$. 
Thus~$q^2 \in K^0$. Since~$q$ commutes with~$K^0$, 
the element~$q^2$ must be in the center of~$K^0$.  
If $K^0 = \SP n / \Z _2$ then~$q^2$ must be the identity, i.e., $q$ 
is an involution. Since the representation of~$K^0$ is irreducible
and $w$ is not in $K^0$, we must have $q=-1$, hence $w=-j$.
 
If $K^0 = \SP n$, then the element $-1$ is contained in $K^0$ 
(it is the only non-trivial involution that can commute with an 
irreducible representation). Hence, the same reasoning as before 
shows that $q^2 = 1$ would imply that $w$ is contained in $K^0$. 
Thus we must have $q^2 =-1$. Then $j$ and $w$ commute with $q$, 
and we finish the proof by taking the complex structure defined by $q$.
\end{proof}
 
\subsection{$G^0$ is covered by $\SU2$}
Assume that $G^0$ is covered by $\SP 1 =\SU 2$.
If $G^0 =\SU 2$ then dimension of $V$ is even and, due to 
\lref{spn}, the involution $w$ must preserve a complex structure.  
Hence its set $F$ of fixed points has even dimension.
The dimension formula $f= \dim(V) - 3+1-1$ provides a contradiction. 
 
Assume now that $G^0= \SO 3$. Then $V=\mathbf R^{2n+1}$ for some $n$.
The dimension formula gives us~$f=2n+1 - 3 +1 -1 =2n-2$.
On the other hand, due to \lref{spn}, the involution $w$ is given by $-w'$ 
for some involution $w' \in \SO3$.
However, the (up to conjugation unique) involution $w'$ in $\SO 3$ 
fixes a subspace of dimension $n$, if $n$ is odd,
and of dimension $n+1$, if $n$ is even.
Thus $f= 2n+1 - n=n+1$ or $f=2n+1 -(n+1)=n$, respectively.
   
By inserting into the previous equation, we deduce that 
$n=3$ or $n=2$. The case $n=2$ is polar (thus excluded) 
and we are left with the case $n=3$.
  
Summarizing, we have shown that if $G^0$ is locally isomorphic to $\SP 1$, 
we must have $G^0 = \SO 3$ and $V= \R^7$.  
However, in this case the action of $G$ on $V$ has cohomogeneity $4$. 
The explicit classification given by \tref{mainclass} shows that 
this case cannot occur as a non-trivial reduction, since
there are no representations of copolarity $3$ listed in 
Table~1.

\subsection{$G^0$ is covered by $\U1\times\SU2$}
The representation space $V$ is the complex space $\C ^n$.
The center of  $G^0$  is a circle $\U 1$ that acts on $\C ^n$ 
as complex multiplication. 
Our nice involution $w$ normalizes  this circle. 
Thus $w$ is either complex linear or complex antilinear.
We have that $f$ is even in the first case and 
$f=n$ in the second case. 

In the complex linear case,  the centralizer 
of~$w$ has dimension~$2$ and the dimension formula
yields
$f=2n-4+2-1=2n-3$. This contradicts the fact that $f$ is even.

In the complex antilinear case, the centralizer 
of~$w$ has dimension~$3$, if~$w$ commutes with 
$\SU2$, and dimension~$1$, if it does not commute. 
In the first case we deduce 
$f=2n-4+3-1=2n-2$ and in the second case $f=2n-4+1-1=2n-4$. 
Using that $f=n$ we derive $n=2$ in the 
first case and $n=4$ in the second case.

The case $n=2$ is polar.
Thus we are left with the case $n=4$.
However, in this case the quotient 
$V/G^0$ has again cohomogeneity $4$. Thus this case is excluded by the 
classification result \tref{mainclass}.

\subsection{ $G^0$ is covered by  $\SU2 \times \SU2$}
There are three cases to be analyzed:
\begin{enumerate}
\item[(I)] $\SU2\cdot\SU2$ acting on $\Q^m\otimes_{\mathbf H}\Q^n$;
\item[(II)] $\SO3\times\SO3$ acting on 
$\R^{2m+1}\otimes_{\mathbf R}\R^{2n+1}$; and
\item[(III)] $\SO3\times\SU2$ acting on $\R^{2m+1}\otimes_{\mathbf R}\Q^n$.
\end{enumerate}

For our nice involution~$w$ 
we have the formula $f=\dim (V) -7 + \dim (C)$.
The conjugation by $w$ can act on $G^0$ as an inner or as an 
outer automorphism. 
We deal with these two cases separately. 

{\bf Outer automorphism.}
We assume first that $w$ acts on $G^0$ not as an inner automorphism. Then $w$ must interchange both factors of $G^0$. Thus we must be in cases~(I) or~(II). 
Moreover, the equality of dimensions  $m=n$ must hold true.

By conjugating~$w$ with an element in $G^0$ we may assume that 
$w$ acts on $G^0$ by interchanging the factors, 
hence by  sending $(g_1,g_2)$ to $(g_2,g_1)$.

Consider the involution $i:V\to V$ defined by interchanging the equal 
factors of $V$: $i(h_1 \otimes  h_2 )= h_2 \otimes h_1$.  
Then $i$ normalizes $G^0$ and the conjugation with $i$
acts on $G^0$ by interchanging the factors of $G^0$. 
Therefore, $w\circ i$ commutes with $G^0$.
Since the representation of $G^0$ is of real type, we must have $w= \pm i$.
Moreover, in both cases, the centralizer of $w$ is the diagonal of $G^0$ that 
has dimension~$3$.

From the dimension formula we conclude that $\dim(V)-f=4$.
In case (II), the dimensions of the sets of fixed points of $i$ and $-i$ 
are given by $(2n+1)(2n+2)/2$ and $(2n+1)(2n)/2$, 
respectively. Since both numbers are larger than $4$ for $n\geq 2$
($n=1$ is polar),  
we get $\dim(V)-f > 4$ and a contradiction.

In case (I),  the dimensions of the sets of fixed points of $i$ and $-i$ are
$n(2n-1)$ and $n(2n+1)$, respectively (note that viewing 
$\Q^n\otimes_{\mathbf H}\Q^{n}$ as the real vector space of quaternionic 
matrices of order $n$, the involution $i$ corresponds to 
transpose conjugation $X\mapsto X^*$). 
Again, for $n>1$ ($n=1$ is polar), 
both numbers are larger than $4$.
Hence $\dim(V)-f$  is larger than $4$ and we derive a contradiction.

{\bf Inner automorphism.}
We assume now that $w$ acts on $G^0$ by an inner automorphism. Then 
$w=q j $, for some $q$ that centralizes $G^0$ and some $j\in G^0$. 

In cases~(I) and~(II) the representations are of real type, 
hence the centralizer of $G^0$ consists of $\pm 1$.   
Since $w$ is not in $G^0$, $q$ must lie outside $G^0$.
In case (I), the element  $-1$ is contained in $G^0$, 
thus we get a contradiction directly.

In case~(II) we must have $q= -1$.  
Thus $w=-j$ and $j$ must be an involution in $G^0$ that is given by a product 
$j=j_1 \cdot j_2$, where each $j_i$ is either an involution or the 
identity in the corresponding factor of $G^0$. 

By the dimension formula, 
the set of fixed points of $j$ must have dimension equal to 
$7-c$, where $c=\dim(C)$.
Denoting by $f_i$ the set of fixed points of $-j_i$  and by 
$e_i$ the set of fixed points 
of $j_i$, we get $7-c=f_1 \cdot f_2 + e_1 \cdot e_2$.

If $j_1$ is not the identity, then $e_1=m$ if $m$ is odd and $m+1$ if $m$ is even. And the corresponding statement is true for $j_2$.

If $j_1=1$  then $c=4$ and we have
$3=(2m+1)e_2$. Hence $e_2=1$ and $m=n=1$ and our representation is polar,
in contradiction to our assumption. Similarly, $j_2=1$ is impossible.

If $j_1$ and $j_2$ are both different from the identity, 
then $c=2$ and in all cases we get $5 > 2mn$.  For $m=n=1$, we have a 
polar representation and, for $m=1$, $n=2$, the dimension formula reads as  $5=1\cdot 3 + 2\cdot 2$,
thus we derive a contradiction.

We are left with the case~(III).
In this case $q^2=j^{-2}$ is contained in the center of
$G^0$. Thus either $q^2 =1$ or $q^2=-1$. If $q^2 =1$ 
then (since $q$ commutes with $G^0$), it must be that
$q=\pm 1$.  But $-1$ is contained in $G^0$, 
hence we would get    $q \in G^0$, which is impossible.

Hence, we must have $q^2 =-1$. Therefore $w$ is a complex involution 
with respect to the complex structure defined by $q$. 
Therefore, $\dim(V)-f$ is an even number. 
However, $c$ is either equal to $2$ or to $4$, hence $7-c$ is odd.  
This provides a contradiction.

\bigskip

\begin{center}
\large\bf
Part 2. Irreducible representations of cohomogeneity $4$ or $5$

\end{center}

\medskip

In this part, we classify the non-polar irreducible representations
of cohomogeneity $4$ or $5$ and prove Theorem~\ref{mainclass}.
Throughout, we use the lists of 
isotropy representations of symmetric spaces~\cite{wolf},
additional polar representations~\cite{eh}, and 
representations of cohomogeneity at most $3$~\cite{hl} 
(see also~\cite{uch,yas,str2}).

\section{Preliminaries}

Let $\rho=(G,V)$ be a real irreducible representation 
of a compact connected Lie group $G$ on a real vector space $V$.
Denote by $c(\rho)$ the cohomogeneity of $\rho$.

\subsection{Basic dimension bound}
It is obvious that 
\[ \dim V \leq \dim G + c(\rho). \]

\subsection{Types}
Recall that, by Schur's lemma,  
the centralizer of $\rho(G)$ in $\mathrm{End}(V)$ is an 
associative real
division algebra, thus, by Frobenius' theorem,
isomorphic to one of $\R$, $\C$ or $\Q$;
accordingly, we say that $\rho$ is of real, complex
of quaternionic type.
There are many alternative characterizations
of such types; the following one is often useful.
$\rho$ is of real type if and only if 
its complexification $\rho^c$ remains irreducible; in this case,
$\rho$ is a real form of a complex irreducible representation.
Otherwise, $\rho^c=\pi\oplus\pi^*$, where $\pi$ is  
complex irreducible and $\rho$ is equivalent to the realification $\pi^r$;
here $\rho$ is of quaternionic (resp. complex) type if $\pi$ and $\pi^*$ are 
equivalent (resp. not equivalent), where $\pi^*$ denotes the dual 
representation. 

\subsection{Quaternionic matrices}
For practical computations, it is often useful to work
with quaternionic matrices. We view a quaternionic vector space $V$
as a right $\Q$-module $V_{\mathbf H}$. 
The set $\mathrm{Hom}(V_{\mathbf H},W_{\mathbf H})$ of 
$\Q$-linear maps $V_{\mathbf H}\to W_{\mathbf H}$ 
is a real vector space; in the special case
$V^*=\mathrm{Hom}(V_{\mathbf H},\Q_{\mathbf H})$, since $\Q$ is a 
bimodule, we get a left $\Q$-module
structure ${}_{\mathbf H}V^*$. 
Now $W_{\mathbf H}\otimes{}_{\mathbf H}V^*$ is well defined and a real vector
space, and there is the usual canonical isomorphism 
$W_{\mathbf H}\otimes{}_{\mathbf H}V^*\cong\mathrm{Hom}(V_{\mathbf H},
W_{\mathbf H})$; this is $G$-equivariant in case $V$, $W$ are 
$G$-representations of quaternionic type. It is also clear that 
$W_{\mathbf H}\otimes{}_{\mathbf H}V^*$ is a real form of 
$W\otimes_{\mathbf C}V^*$. Having said this, henceforth
we write just $W\otimes_{\mathbf H}V$ for 
$W_{\mathbf H}\otimes{}_{\mathbf H}V^*$.

\subsection{Slices and sums}\label{sec:slices-sums}
The slice representation at $v\in V$ is the induced representation
of the isotropy group $G_v$ on the normal space $N_v(Gv)$. 
It is known that the cohomogeneity of a slice representation is equal 
to the cohomogeneity of the original representation. This works as 
an inductive argument allowing one to compute precisely
the cohomogeneity (compare~\cite{hh}). 
One can also apply this remark to sums. 
Let $\rho=\rho_1\oplus\rho_2$ be the 
representation $(G,V_1\oplus V_2)$. Let $v_1\in V_1$ be a regular point
for $(G,V_1)$. Then consideration of the slice of $\rho$ at $v_1$ yields 
\[ c(\rho)= c(\rho_1) + c(G_{v_1},V_2). \]
In particular, $c(\rho)\geq c(\rho_1)+c(\rho_2)$ and equality
holds if and only if $(G_{v_1},V_2)$ is orbit equivalent to 
$\rho_2=(G,V_2)$.

\subsection{Tensor products}\label{sec:subgroups-products}

It is convenient to introduce the following notation.
If $A$ and $B$ are classical groups, we shall denote by $A\otimes B$ 
the irreducible representation given by the tensor product of the 
vector representations (the field over which the tensor 
product is being taking is determined by the 
types of the factor-representations).

In general, for a tensor product 
$\rho=(G=G_1\times G_2,V_1\otimes_{\mathbf F}V_2)$, where 
$\dim V_i=n_i$ and
$\mathbf F=\R$, $\C$ or $\Q$, we respectively have
\[\begin{array}{l}
\rho(G)\subset\SO{n_1}\otimes\SO{n_2};\ \mbox{or}\\
\rho(G)\subset\U{n_1}\otimes\U{n_2};\ \mbox{or}\\
\rho(G)\subset\SP{n_1}\otimes\SP{n_2}.\\
\end{array}\]
It follows that 
\begin{equation*}
c(\rho)\geq\min\{n_1,n_2\}.
\end{equation*}

The following monotonicity lemma will be used to simplify the
estimation of the cohomogeneity of some representations. 
\begin{lem}\label{monoton}
Let $\rho_1=(G_1,V_1)$ be an arbitrary real representation, 
and let $\rho_2(n)=(G_2(n),\mathbf F^n)$
be the standard representation of $\OG n$, $\U n$, $\SP n$ on
$\R^n$, $\C^n$, $\Q^n$, respectively. Assume $\rho_1$ is of
$\tilde{\mathbf F}$-type for $\tilde{\mathbf F}\subset\mathbf F$ and 
consider the tensor product $\rho(n)=(G(n)=G_1\times G_2(n),
V_1\otimes_{\tilde{\mathbf F}}\mathbf F^n)$. Then 
$c(\rho(n))\leq c(\rho(n+1))$. 
\end{lem}

\Pf Since the cohomogeneity is the topological dimension 
of the orbit space, it is enough to show that the orbit 
space of $\rho(n)$ injects into that of $\rho(n+1)$. Since $V_1$ 
is a real representation, we can identify $V_1$ with $V_1^*$ and replace
the tensor product by the space of linear maps 
$\mathcal L_n=\mathrm{Hom}_{\tilde{\mathbf F}}(V_1,\mathbf F^n)$. 
Now $(g_1,g_2)\in G_1\times G_2(n)$ acts on $A\in \mathcal L_n$ 
by mapping it to $g_2Ag_1^{-1}$.
We consider the standard embedding $\mathcal L_n\to \mathcal L_{n+1}$.  
To prove the desired injectivity, we just
need to show that any two elements $A$, $B\in\mathcal L_n$ that
are in the same $G(n+1)$-orbit must be in the same $G(n)$-orbit. 
It is obvious that we can restrict to the case $G_1=\{1\}$.
In this case $A$ and $B$ are respectively given by $m$-tuples
$(a_1,\ldots,a_m)$, $(b_1,\ldots,b_m)$ of elements in 
$\mathbf F^n\subset\mathbf F^{n+1}$
where $m=\dim_{\mathbf{\tilde F}} V_1$. 
Let $g_2\in G_2(n+1)$ such that $g_2A=B$. 
One restricts $g_2$ to the subspace of $\mathbf F^n$ spanned 
by $a_1,\ldots,a_m$ and then extends it to an element of $G_2(n)\subset
G_2(n+1)$. 
\hfill\mbox{}\EPf

\begin{cor}
We have $c(G_1\otimes\SU n)\leq c(G_1\otimes\SU{n+1})+1$. 
\end{cor}

\subsection{Slices of products}\label{sec:slice-prod}

Let $\rho_i$ be a real irreducible representation of a Lie
group $G_i$ on $V_i$, where $i=1$, $2$. Then 
$\rho=\rho_1\otimes_{\mathbf R}\rho_2$ is a real  
representation of $G=G_1\times G_2$ on $V=V_1\otimes_{\mathbf R}V_2$
which is irreducible if at least one of the $\rho_i$ is of real type.
Let $v=v_1\otimes v_2\in V$ be a pure tensor
where $v_i\in V_i$ is regular. Then $G_{v_i}$ is a
principal isotropy subgroup of $\rho_i$ and
the connected isotropy group $(G_v)^0$ equals $H=H_1\times
H_2$, where $H_i=(G_{v_i})^0$. Since the normal space
\[ N_v(Gv)=\R v\oplus\left[v_1\otimes (N_{v_2}(G_2v_2)\ominus\R v_2)\right]
\oplus\left[(N_{v_1}(G_1v_1)\ominus\R v_1)\otimes v_2\right]
\oplus\left[v_1^\perp\otimes v_2^\perp\right], \]
by considering the slice representation at $v$, we get
\[ c(\rho) = c(H,v_1^\perp\otimes v_2^\perp) 
+ c(\rho_1) + c(\rho_2) -1. 
\end{equation*}
We will mostly use this remark in case $c(\rho_i)=1$ for $i=1$, $2$;
then the decomposition of $v_1^\perp\otimes v_2^\perp$ into 
irreducible components is easier. 

In the cases of complex and quaternionic tensor products, one uses
similar though slightly more complicated arguments that we 
explicit below in the individual cases.

\subsection{Cohomogeneity $c(\rho)=1$}  \label{transac}
The list of transitive linear actions on spheres 
is well known (cf.~Theorem of Borel-Montgomery-Samelson).
We include it for easy reference. 
Here and below, 
$\rho$ always denotes a real representation, so e.g.~$\rho=(\SU n,\C^n)$
means the realification of $\C^n$. On the other hand, if $W$ is complex,
$[W]_{\mathbf R}$ denotes a real form.
\[\begin{array}{|c|c|c|c|}
\hline
G & \rho & \mbox{\textsl{Real dim}} & \mbox{\textsl{Remarks}}\\
\hline
 \SO n & \R^n & n & -\\
 \SU n & \C^n & 2n & - \\
 \SP n & \C^{2n} & 4n & -\\
 \G & \R^7 & 7 & -\\
 \Spin7 & \R^8 & 8 & \mbox{spin representation}\\
  \Spin9 & \R^{16} & 16 & \mbox{spin representation}\\
 \SU n\cdot\U1 & \C^n\otimes_{\mathbf C}\C & 2n & - \\
 \SP n\cdot\U1 & \Q^{n}\otimes_{\mathbf C}\C & 4n & - \\
 \SP n\cdot \SP1 & \Q^{n}\otimes_{\mathbf H}\Q & 4n & - \\
\hline
\end{array} \]

\subsection{$G$ is simple and $2\leq c(\rho)\leq8$}\label{sec:simple}
A lemma of Onishchik~\cite[Lemma~3.1]{On} explains that the dimension of 
a complex irreducible representation is an increasing function
of the highest weight, where one uses a partial order
naturally defined on the set of dominant integral weights. 
Using this lemma, it is a matter of patience to list real 
irreducible representations of low cohomogeneity of simple groups.
In the tables below we go up to cohomogeneity $8$ (the tables 
in~\cite[ch.~I, \S2]{hh} are also helpful). Up to a few cases,
our list can also be obtained from Lemma~2.6 in~\cite{K}.

\subsubsection{Polar representations}\label{sec:polar}

\[\begin{array}{|c|c|c|c|c|c|}
\hline
G & \rho & \mbox{\textsl{Conditions}} & c(\rho) & \mbox{\textsl{Type}}&
\mbox{\textsl{Symmetric space}}\\
\hline
 \SO n & S^2_0(\R^n) & 3\leq n\leq 9 & n-1 & r & \SU n/\SO n\\
 \SP n & [(\Lambda^2\C^{2n}\ominus\C]_{\mathbf R} &  3\leq n\leq 9 & n-1 &
r & \SU{2n}/\SP n\\
 \F & \R^{26} & - & 2 & r & \E6/\F\\
 \SU n & [\C^n\otimes_{\mathbf C}\C^{n*}\ominus\C]_{\mathbf R}
& 3\leq n\leq 9 & n-1 & r & \mbox{Adjoint}\\
 \SO n &  \Lambda^2\R^n
& 5\leq n\leq17 & \left[\frac n2\right] & r &
\mbox{Adjoint}\\ 
 \SP n & [S^2(\C^{2n})]_{\mathbf R} & 2\leq n\leq8 & n & r &\mbox{Adjoint}\\
 \F & \R^{52} & - & 4 & r & \mbox{Adjoint} \\
 \G & \R^{14} & - & 2 & r & \mbox{Adjoint}  \\
 \E6 & \R^{78} & - & 6  & r & \mbox{Adjoint} \\
 \E7 & \R^{133} & - & 7  & r & \mbox{Adjoint} \\
 \E8 & \R^{248} & - & 8  & r & \mbox{Adjoint} \\
 \SU n & \Lambda^2\C^n & 5\leq n\leq17, \mbox{$n$ odd} & \frac{n-1}2 & c & 
\SO{2n}/\U n\\
 \Spin{10} & \C^{16} & - & 2 & c & \E6/(\U1\Spin{10})\\
 \SP4 & [\Lambda^4\C^8\ominus\Lambda^2\C^8]_{\mathbf R} 
& - & 6 & r & \E6/\SP4\\
 \SU8 & [\Lambda^4\C^8]_{\mathbf R} & - & 7 & r & \E7/\SU8 \\
\Spin{16} & \R^{128} & - & 8 & r & \E8/\Spin{16}\\
\hline
\end{array} \]

\subsubsection{Non-polar representations}\label{sec:nonpolar}

\[\begin{array}{|c|c|c|c|c|}
\hline
G & \rho & \mbox{\textsl{Conditions}} & c(\rho) & \mbox{\textsl{Type}}\\
\hline
 \SO3 & \R^n & n=7,9,11 & n-3 & r\\
 \SU2 & \C^4 & - & 5 & q \\
 \SU6 & \Lambda^3\C^6 & - & 7 & q\\
 \SU n & \Lambda^2\C^n & 6\leq n\leq 14, \mbox{$n$ even} & \frac n2+1 & c \\
 \SU n & S^2\C^n & 3\leq n\leq 7 & n+1 & c \\
 \SP3 & \Lambda^3\C^6\ominus\C^6 & - & 7 & q\\
 \Spin{12} & \C^{32} & - & 7 & q \\
 \E6 & \C^{27} & - & 4 & c\\
 \E7 & \C^{56} & - & 7 & q\\
\hline
\end{array} \]

\section{Products of mixed type}\label{sec:mixed}

Having already discussed simple groups, we next turn to 
the case of representations $\rho$ of 
non-simple groups $G$ with $c(\rho)=4$ or $5$. 
We roughly divide the discussion according to 
the types of the factors and start with the case of mixed type.

\begin{lem}\label{lem:so-sp}
We have $c(\SO m\otimes \SP n)\geq11$ for $m\geq3$ and $n\geq2$, 
and $c(\SO m\otimes \SP1)\geq7$ for $m\geq4$.
Moreover $c(\SO3\otimes\SP1)=6$ and $c(\SO3\otimes\U2)=5$.
\end{lem}

\Pf By dimensional reasons, $c(\SO3\times\SP2)\geq11$ and 
$c(\SO4\times\SP1)\geq7$, so the first two assertions 
follow from Lemma~\ref{monoton}. The cohomogeneity of $\SO3\otimes\U2$
can be directly computed by using slices and noticing that 
the principal isotropy group is trivial, and then the other
cohomogeneity follows. \EPf

\begin{lem}\label{lem:su-so}
For $m\geq3$, $n\geq3$ we have $c(\SU m\otimes \SO n)\geq 7$
unless $m\geq4$ and $n=3$.
Moreover $c(\U m\otimes\SO3)=6$ for $m\geq3$.
\end{lem}

\Pf By dimensional reasons, $c(\U3\otimes\SO4)\geq9$.
Hence $c(\SU m\otimes\SO n)\geq9$ if $m\geq3$ and $n\geq4$. 

The case $n=3$ is dealt with separately:
\begin{eqnarray*}
 c(\SU m\otimes\SO3)&=&1+c(\SU{m-1}\times\SO2,i\R^2\oplus\C^{m-1}\otimes_{\mathbf C}\C^2)\\
&=&2+c(\SU{m-1},2\C^{m-1})\\
&=&\left\{\begin{array}{ll}
         7, & \mbox{if $m=3$,} \\
         6, & \mbox{if $m>3$.}
       \end{array}\right. 
\end{eqnarray*}
A similar computation shows that  
$c(\U m\otimes\SO3)=6$ for $m\geq3$. \EPf

\begin{lem}\label{lem:su-sp}
We have:
\begin{enumerate}
\item $c(\U m\otimes\SP2)\geq6$ if $m\geq4$;
\item $c(\U3\otimes\SP2)=5$ and $c(\SU3\otimes\SP2)=6$;
\item $c(\SU m\otimes \SP n)\geq7$ if $m\geq3$ and $n\geq3$. 
\end{enumerate}
\end{lem}

\Pf By dimensional reasons, $c(\U4\otimes\SP2)\geq6$, 
$c(\U4\otimes\SP3)\geq11$ and $c(\SU3\otimes\SP3)\geq7$.
Hence~(a) and~(c) follow from Lemma~\ref{monoton}.

On the other hand, one computes directly that 
$c(\U 3\otimes \SP2)=5$, from which follows 
$c(\SU 3\otimes \SP2)=6$. This proves~(b). \EPf

\section{The case $G=G_1\times G_2$, $V=V_1\otimes_{\mathbf R}V_2$}
\label{sec:real-prod}

Here $G_1$ and $G_2$ are non-necessarily simple,
$\rho_i=(G_i,V_i)$ are real irreducible representations and at least one
of them is of real type. Let $m=\dim_{\mathbf R} V_1\leq 
n=\dim_{\mathbf R} V_2$. Recall that $\SO m\otimes\SO n$ is polar.
Since $c(\rho)=4$ or $5$, in view 
of Subsection~\ref{sec:subgroups-products} we may assume that $2\leq m\leq4$. 
Moreover, owing to the next lemma, $c(\rho_1)=c(\rho_2)=1$. 

\begin{lem}
Let $\rho_1=(G_1,V_1)$ and $\rho_2=(G_2,V_2)$ be 
real irreducible non-trivial representations 
(non-necessarily of real type)
and consider their real tensor 
product $\rho=(G,V)$, where $G=G_1\times G_2$ and 
$V=V_1\otimes_{\mathbf R}V_2$.
If either $\rho_1$ or  $\rho_2$ has cohomogeneity 
bigger than one, then $c(\rho)\geq6$. 
\end{lem}

\Pf Fix $v=v_1\otimes v_2\in V$ where $v_i\in V_i$ is 
$G_i$-regular. Write $H=(G_v)^0$ and $H_i=((G_i)_{v_i})^0$.
Then $H=H_1\times H_2$. Without loss of generality, we 
may assume that $c(\rho_1)=1$ and $n=c(\rho_2)\geq2$.
Then 
\[ c(\rho)=n+c(H,v_1^\perp\otimes v_2^\perp). \]
Put $U=T_{v_2}(G_2v_2)$. Then $U\neq0$ and
\[ v_1^\perp\otimes v_2^\perp=(n-1)v_1^\perp\otimes\R\oplus v_1^\perp\otimes U. \]
If $n\geq3$, then it follows from the above 
that $c(\rho)\geq6$. Suppose $n=2$. 
Then $\rho_2$ is polar and $G_2(v_2)$ is an isoparametric 
submanifold of $V_2$. Since $\rho_2$ is irreducible, 
$G_2(v_2)$ has at least $3$ distinct principal curvatures 
so $U$ has at least $3$ $H_2$-irreducible components. 
Thus $c(H,v_1^\perp\otimes U)\geq3$ and it follows that 
\[ c(H,v_1^\perp\otimes v_2^\perp)\geq4 \]
and
\[ c(\rho)\geq 2+4=6, \]
as desired. \EPf

\subsection{$m=2$}

Then $\rho_1=(\SO2,\R^2$). Also, $\rho_2=(G_2,V_2)$ 
is of real type.
Running through the cases in which $c(\rho_2)=1$, the only non-polar
$\rho$ that we get are $(\SO2\times \Spin9,\R^2\otimes\R^{16})$,
and $(\SO2\times \SP n\SP1,\R^2\otimes(\Q^n\otimes_{\mathbf H}\Q))$ 
where $n\geq2$, both of which have 
$c(\rho)=3$ and copolarity $1$~\cite[Prop.~7.12]{gt0}. 

\subsection{$m=3$ or $4$}\label{so4spin7}
Assume first that $\rho_1=(\SO m,\R^m)$.
Running through the cases in which $\rho_2$ is of real type
and using the dimension bound,
we get that $\rho$ is one of $(\SO3\times \G,\R^3\otimes\R^7)$, 
$(\SO3\times \Spin7,\R^3\otimes\R^8)$ or
$(\SO4\times \Spin7,\R^4\otimes\R^8)$.
(Note 
that $c(\SO m\otimes \SP n\SP1)\geq8$ for $m\geq3$ and $n\geq2$
by Lemma~\ref{lem:so-sp}.)
The second representation
is   orbit equivalent to $(\SO3\times\SO8,\R^3\otimes\R^8)$
and hence polar~\cite{eh}. The first one has $c(\rho)=4$ and copolarity 
$2$~\cite[Lemma 6.11]{gt}. The last one has $c(\rho)=5$ and 
copolarity~$3$, which can be checked via a similar reasoning as 
in~\cite[Lemma 6.11]{gt} and using Lemma~7.13~\emph{loc.cit.}
If $\rho_2$ is not of real type, we invoke Section~\ref{sec:mixed}
to get one more example, namely $\SO3\otimes\U2$. 

Finally, the only other possibilities for $\rho_1$ are that 
it equals $(\SU2,\C^2)$ or $(\U2,\C^2)$. Then $\rho_2$ is of real
type; again this case is covered by Section~\ref{sec:mixed}
and we get no other examples. 

\section{The case $G=G_1\times G_2$, $V=V_1\otimes_{\mathbf H}V_2$}
\label{sec:quat-prod}

In principle, $G_1$ and $G_2$ could be non-simple.
The $\rho_i=(G_i,V_i)$ are real irreducible representations and both 
of them are of quaternionic type. The first remark is that 
we may assume that $G_1$ and $G_2$ are simple.
Indeed, if, say $G_2$, is not simple, then $V_2$ is a 
tensor product $W_1\otimes_{\mathbf R} W_2$ where 
$W_1$ is of real type and $W_2$ is of quaternionic type;
by rearranging the factors in $V$ we see that this case is included
in Section~\ref{sec:real-prod}. 
Recall that $\SP m\otimes\SP n$ is polar, so
this case will be omitted in the sequel.

Assume $G_1=\SP1$. If $V_1=\mathbf H$ then we must have 
$c(\rho_2)\leq c(\rho)+\dim\SP1 \leq 8$. 
Since $G_2$ is simple, running through the list in section~\ref{sec:nonpolar}
we get five  polar representations $\rho$ associated to irreducible
quaternionic-K\"ahler symmetric spaces.

For $m\geq 3$, we have  
$c(\SP n, m \cdot \mathbf H ^n) >8$ for $n\geq1$.  
Therefore $c(\SP1 \times G_2 , \mathbf H ^m \otimes _{\mathbf H} V_2 ) \geq 
c(\SP1 \times \SP n , \mathbf H ^m \otimes_{\mathbf H}  \mathbf H^n ) \geq
c(\SP n , m \cdot \mathbf H ^n ) -3 >5$.  
Moreover, if $c(\rho_2) >1$ then $c(\rho_2) \geq 5$, due to
Subsection ~\ref{sec:nonpolar}. Hence $c(2\rho_2) \geq 10$ and 
$c(\SP1 \times G_2, \mathbf H^2  \otimes_{\mathbf H} V_2) \geq 7$.  

In the case $G_1 =\SP 1$, we are thus
left with the representations 
$(\SP 1 \times \SP n , \mathbf H^2 \otimes_{\mathbf H} \mathbf H^n)$,
but they have cohomogeneity $3$ and copolarity~$1$~\cite[Prop.~7.12]{gt0}.

Thus we may assume that both groups $G_i$ are not equal to $\SP 1$.
Let $m=\dim_{\mathbf H} V_1\leq n=\dim_{\mathbf H} V_2$.
Since  $\SP m\otimes\SP n$ is polar  and we are interested in the cases   
 $c(\rho)=4$ or $5$,  we may assume that  $m\leq 4$. Hence $2\leq m\leq 4$
and we deduce that $\rho_1(G_1)=\SP m$.

We rule out those cases by showing $c(\rho)\geq6$ as follows. 
In view of Lemma~\ref{monoton}, we may assume $m=2$. 
Now if $c(\rho_2) \geq 8$, then
the same argument as above, together with $\dim (G_1)=10$, 
implies that $c(\rho) \geq 2\cdot 8-10 =6$. Hence it remains
to check for $\rho_2$ equal to one of 
the four representations in the table in Subsection~\ref{sec:nonpolar}.  
In all cases, we finish by using the basic dimension bound.

\section{The remaining case: $G=G_1\times G_2$, $V=V_1\otimes_{\mathbf C}V_2$}
\label{sec:cx-prod}

For $\rho=(G,V)$ with $G$ non-simple,
we may assume that all tensor products in $V$ are over $\C$
due to Sections~\ref{sec:real-prod} and~\ref{sec:quat-prod}.
It follows that there are no factors of real type and at most one factor
of quaternionic type. 
It is also useful to recall that $\SU m\otimes\U n$ is always polar
and $\SU m\otimes\SU n$ is polar if and only if $m\neq n$; otherwise, it
has cohomogeneity $m+1$ and copolarity $m-1$~\cite[\S~3.3]{got}.

Consider first the case $\rho=(\U1\times G_2,\C\otimes_{\mathbf C}V_2)$
where $G_2$ is simple, $\rho_2=(G_2,V_2)$ is real irreducible
of complex or quaternionic type
and~$c(\rho_2)>1$. A glance at the tables in Section~\ref{sec:simple} yields
only one example, namely, $(\U2,\C^4)$. 
We will show there is only one further example, with cohomogeneity~$5$, 
a circle factor and trivial principal isotropy group.  
Hence we may assume in the sequel that a circle factor is always present.

Assume the factor of quaternionic type
is present. We consider first the case 
$\rho(G)\subset\SU2\otimes\SU m\otimes\U n=
\SU2\otimes\U m\otimes\SU n$ where 
$3\leq m\leq n$. It follows from Lemma~\ref{monoton} that 
$c_{m,n}=c(\SU2\otimes\SU m\otimes\U n)$ increases with $m$ and $n$,
but note that $c_{3,3}\geq16$ by dimensional reasons. 

Consider next the case $\rho$ is of the form
$(\U2\times G_2,\C^2\otimes_{\mathbf C}V_2)$, where $\rho_2=(G_2,V_2)$
is real irreducible of complex type with $G_2$ simple. 
We have $c(\rho)\geq c(2\rho_2)-4>2c(\rho_2)-4\geq6$ for $c(\rho_2)\geq5$.
Otherwise $2\leq c(\rho_2)\leq4$ and we can use the
dimension bound to rule out the few possibilities given 
in the tables in Section~\ref{sec:simple}.

Hence if the factor of quaternionic type
is present, the discussion in Section~\ref{sec:mixed} yields that
$\rho$ is $\U3\otimes\SP2$.

Finally we consider the case in which all factors are of complex type. 
The case $\rho(G)\subset\SU m\otimes\SU n\otimes\U p$
where $3\leq m\leq n\leq p$ is discarded by means 
of Lemma~\ref{monoton} as above. Hence we may assume that $G$ is the 
product of
two simple groups and a circle factor.
Now $\rho(G)\subset\U m\otimes\SU n=\SU m\otimes\U n$
and we may assume $m\leq n$. 
Since we are interested in $c(\rho)=4$ or $5$, in view of 
Subsection~\ref{sec:subgroups-products} we may assume $m=3$ or $4$.
Again invoking Lemma~\ref{monoton}, it suffices to exclude
the case $\rho=(\U3\times G_2,\C^3\otimes_{\mathbf C}V_2)$
where $G_2$ is simple. In this case
$c(\rho)\geq c(3\rho_2)-9>3c(\rho_2)-9\geq6$ if $c(\rho_2)\geq5$. 
On the other hand, for $2\leq c(\rho_2)\leq4$ the few cases given in the tables
in Section~\ref{sec:simple} are excluded
by the dimension bound.

\section{Final arguments} 

We are going to finish the proof of Theorem~\ref{mainclass}. 
 
Collecting the lists of Subsection~\ref{sec:simple}
as well as the results of Sections~\ref{sec:real-prod}, 
\ref{sec:quat-prod} and~\ref{sec:cx-prod}
yields the first two columns of the Tables~1 and~2.  
It remains to verify the statements about the 
copolarities and the boundaries. 

All but the first two representations in cohomogeneity $4$ have been 
shown to admit reductions to minimal 
generalized sections with the torus $T^2$ 
as the identity component of the group acting on it~\cite{got}. 
Since these actions admit reductions, their quotients have non-empty boundary.
On the other hand, for the first two representations, the quotients do 
not have boundary (owing to Section~\ref{connectedcase},
which is independent of Theorem~\ref{mainclass}).  
Due to \pref{boundaryexist}, these representations do not admit reductions.

In the same way all but two representations of cohomogeneity $5$ admit 
reductions to minimal generalized sections with $3$-dimensional
tori as the identity component of the group acting on it 
(see Subsection~\ref{so4spin7} and~\cite{got}), and the 
orbit space of $\SU2$ on $\C^4$ has no boundary.

For the remaining two presentations, one is a reduction of the 
second via a generalized section. Namely, for the representation
of $\U3\times\SP2$ we enlarge the group by adjoining 
complex conjugation of matrices to obtain an orbit equivalent 
representation. The new representation has non-trivial principal
isotropy group whose fixed point, given by
the subspace of real matrices, 
is a generalized section and the representation space 
of $\SO 3 \times  \U 2$.  Therefore both quotients are 
isometric and have non-empty boundaries.
It remains to prove that the representation 
$\rho$ of $\SO 3 \times  \U 2$ on $\R ^3 \otimes  \R ^4$
has trivial copolarity. 
Since
this information is not needed in the proof of Theorem~\ref{special},
we can use this theorem  to show that the \emph{abstract} copolarity 
of $\rho$ is trivial as follows. 
The dimension of  $\SO 3 \times  \U 2 $ is $7$. If it had a 
minimal   reduction
to a representation of a group $G$ with $\dim  (G) \leq 6$, 
then the restricted representation of $G^0$ would be 
reducible, due to \pref{insertprop}.  
Hence one could apply \tref{special} to deduce that
the representation of $\SO 3 \times  \U 2 $  
must have copolarity equal to the abstract copolarity equal to~$3$.
One could either argue directly that $\rho$ does not admit a reduction to 
a finite extension of its maximal torus, or, more elegantly, 
use \cref{finalpoint} to deduce that $\rho$ must have non-trivial principal
isotropy groups. This provides a contradiction
and finishes the proof of Theorem~\ref{mainclass}.

\providecommand{\bysame}{\leavevmode\hbox to3em{\hrulefill}\thinspace}
\providecommand{\MR}{\relax\ifhmode\unskip\space\fi MR }
\providecommand{\MRhref}[2]{%
  \href{http://www.ams.org/mathscinet-getitem?mr=#1}{#2}
}
\providecommand{\href}[2]{#2}


\end{document}